%% file: relative_fp.tex
\documentclass{amsart}

\input xy
\xyoption{all}
\usepackage{lacromay}
\usepackage{notation}
\usepackage{autoref_long}

\usepackage{amsmath,amssymb,amsthm}
\usepackage{graphics}
\usepackage{enumerate}
\usepackage{mathtools}
\usepackage{hyperref}

\CompileMatrices

\theoremstyle{theorem}
\newtheorem*{thmnonumber}{Theorem}
\newtheorem*{thmA}{Theorem A}
\newtheorem*{thmB}{Theorem B}

\newcommand{\rcomcat}[2]{\Pi_0({#1},{#2})}
\newcommand{\rcomcatf}[2]{\Pi_0^f({#1},{#2})}
\newcommand{\comcat}[1]{\Pi_0({#1})}
\newcommand{\comcatf}[1]{\Pi_0^f({#1})}
\newcommand{\rcomsp}[2]{\overline{{#1}|{#2}}}
\newcommand{\rcomspm}[1]{\overline{{#1}}}
\newcommand{\rfuncat}[2]{\Pi_1({#1},{#2})}
\newcommand{\rfuncatf}[2]{\Pi_1^f({#1},{#2})}

\newcommand{\funcatf}[1]{\Pi_1^f({#1})}
\newcommand{\relidx}[3]{i_{{#2}|{#3}}({#1})}
\newcommand{\rlnumb}[3]{L_{{#2}|{#3}}({#1})}
\newcommand{\alnumb}[2]{L_{#2}({#1})}
\newcommand{\runicov}[2]{\widetilde{{#1}|{#2}}}
\newcommand{\rrgeo}[3]{\,^{ge}\hspace{-.6ex}R_{{#2}|{#3}}({#1})}
\newcommand{\rralg}[3]{\,^{gl}\hspace{-.7ex}R_{{#2}|{#3}}({#1})}
\newcommand{\ralg}[1]{\,^{gl}\hspace{-.7ex}R({#1})}
\newcommand{\rrkw}[3]{\,^{KW}\hspace{-.7ex}R_{{#2}|{#3}}({#1})}
\newcommand{\rrhtpy}[3]{\,^l\hspace{-.7ex}R_{{#2}|{#3}}({#1})}

\begin{document}

\title{Relative Fixed Point Theory}
\author{Kate Ponto}
\date{\today\\
The author was supported by a National Science Foundation postdoctoral fellowship.}
\maketitle

\input{relative_fp_00}

\input{relative_fp_0}
\input{relative_fp_1_2}
\input{relative_fp_3}
\input{relative_fp_4}

\input{relative_fp_5}
\input{relative_fp_6}
\input{relative_fp_8}
\input{relative_fp_9}

\bibliographystyle{plain.bst}
\bibliography{trace}
\end{document}

%% file: relative_fp_00.tex
\section*{Abstract}

The Lefschetz fixed point theorem and its converse have many generalizations.
One of these generalizations is to endomorphisms of a space relative to a fixed 
subspace.  
In this paper we define relative Lefschetz numbers and 
Reidemeister traces using traces in bicategories with shadows.  
We use the functoriality of this trace to identify different forms 
of these invariants and to prove a relative Lefschetz fixed point 
theorem and its converse.

%% file: relative_fp_0.tex
\section*{Introduction}

The goal of topological fixed point theory is to find invariants that
detect if a given endomorphism of a space has any fixed points.  The Lefschetz fixed
point theorem identifies one such invariant.

\begin{thmnonumber}[Lefschetz fixed point theorem] Let $B$ be a closed smooth manifold
and \[f\colon B\rightarrow B\] be a continuous map.  If $f$ has no fixed points then the 
Lefschetz number of $f$ is zero.
\end{thmnonumber} 

The Lefschetz number is the alternating sum of the levelwise traces of the map
induced by $f$ on the rational homology of $B$.  This is a relatively computable 
invariant.  It gives a necessary, but not sufficient, condition for determining if a 
continuous map does not have any fixed points.

If we put additional restrictions on the map $f$, such as requiring it to 
preserve a subspace of $B$, the Lefschetz number still gives a necessary condition
for $f$ to be fixed point free.  However, this invariant ignores the relative structure and
so is not the best possible invariant.  For example, the identity map of the circle is
homotopic to a map with no fixed points and so the Lefschetz number is zero.  
If this map is required to preserve a proper subinterval it is no longer homotopic 
to a map with no fixed points.

There is a refinement of the Lefschetz number defined 
using the induced maps on the rational homology of the subspace and the relative 
rational homology.

\begin{thmA}[Relative Lefschetz fixed point theorem]\label{thma} Let $A\subset B$ be closed
smooth manifolds and \[f\colon B\rightarrow B\] be a continuous map such that $f(A)\subset A$.
If $f$ has no fixed points then the relative Lefschetz number of $f$ is zero.
\end{thmA}

The relative Lefschetz number of the
identity map of the circle relative to a proper subinterval is not zero. 

Both of these theorems give a condition that implies
that a continuous endomorphism 
\[f\colon B\rightarrow B\] has a fixed point.  In most cases they do not 
give a condition that implies $f$ has no fixed points.  To address this question
a refined invariant and some restrictions on the spaces have to be introduced.  This 
refined invariant is called the Nielsen number or the Reidemeister trace.

\begin{thmnonumber}[Converse to the Lefschetz fixed point theorem]
Let $B$ be a closed smooth manifold of dimension at least 3 and \[f\colon
B\rightarrow B\] be a continuous map.  The map $f$ is homotopic 
to a map with no fixed points if and only if the Reidemeister trace of $f$ 
is zero.
\end{thmnonumber}

The Reidemeister trace is a partitioning of the Lefschetz number to reflect
the ways fixed points can be changed by a homotopy of the original map.
There is a generalization of the Reidemeister trace to a relative Reidemeister
trace that is very similar to the generalization of the Lefschetz number to the 
relative Lefschetz number.

\begin{thmB}[Converse to the Relative Lefschetz fixed point theorem]\label{thmb}
Suppose \\$A\subset B$ are closed smooth manifolds of dimension at least 3 and
the codimension of $A$ in $B$ is at least 2.
A map \[f\colon B\rightarrow B\] such that $f(A)\subset A$ is homotopic to 
a map with no fixed points through maps preserving the subset $A$
if and only if the relative Reidemeister trace of $f$ is zero. 
\end{thmB}

The goal of this paper is to provide definitions of the relative Lefschetz number
and relative Reidemeister trace and proofs of Theorems A and B
that satisfy several requirements.  First, the relative Reidemeister trace should 
detect if a map is relatively homotopic to a map with no fixed points.  It is not necessary
for the relative Reidemeister trace to provide a lower bound for the number of fixed point of a given
map.  Second, the invariants should
be trace-like.  This means that they can be described using the duality and trace in
bicategories defined in \cite{thesis}.  The relative Reidemeister trace
should to be compatible with the approach of \cite{KW, KW2}.  Those papers 
give a proof of the converse to the  Lefschetz fixed point theorem that is different from the standard 
simplicial proof.  Finally, the relative Reidemeister trace should be compatible with 
an equivariant generalization of the Reidemeister trace described in \cite{equiv}.  

While the Lefschetz number and the Reidemeister trace have long established
definitions, the relative forms of these invariants are less settled.  Versions of the 
relative Lefschetz number have been defined in \cite{Bowszyc, Jezierski95}
and of the relative Reidemeister trace in \cite{NOW94,schirmer86, Zhao01, Zhao05,
Zhao06}.  The papers \cite{schirmer86, Zhao01, Zhao05} are primarily interested 
in determining lower bounds for 
the number of fixed points and so are generalizations of the Nielsen number.
The invariants defined in \cite{NOW94,Zhao06} are more trace-like, but the 
definition are still motivated by connections to the Nielsen number.  None of these
invariants satisfy all of our conditions above, and none of them exactly coincide
with the definitions given here.

In this paper we give proofs of Theorems A and B following the outline of \cite{thesis}.
We use duality and trace in bicategories with shadows to define two forms of 
the relative Lefschetz number and
the relative Reidemeister trace.  Then, using functoriality, we show 
different invariants coincide.  Finally, we generalize
Klein and Williams's proof of the converse to the Lefschetz fixed point
theorem in \cite{KW} to complete the proof of the converse to the relative 
Lefschetz fixed point theorem.

In the first two sections we will recall the necessary definitions of duality and trace 
in symmetric monoidal categories and in bicategories with shadows.  In 
\autoref{bicatdist} we will describe some examples of bicategories with shadows and generalize
some results from \cite{thesis} that describe specific examples of duals.

In \autoref{relidxsec}  we apply this category theory to the relative 
Lefschetz number. 
In Sections \ref{geosec} and  \ref{algsec} we define the relative 
Reidemeister trace.  We describe how to compare these invariants to each other
and how to compare them to the relative Nielsen number defined by 
Schirmer and Zhao.   In Sections \ref{converse} and \ref{relsec} we give a proof of the converse
to the relative Lefschetz fixed point theorem following the proofs given by Klein and
Williams in \cite{KW,KW2}.  In \autoref{extracat} we include some formal results 
omitted from the third section.

We assume the reader is familiar with the basic definitions of Nielsen theory.  References
for this topic include \cite{Brownbook, Jiang}.

\subsection*{Acknowledgements}
I would like to thank  Peter May, Gun Sunyeekhan, and Bruce Williams
for many helpful conversations and comments on previous drafts of this paper.
I would also like to thank Xuezhi Zhao for sending me one of his preprints.

\subsection*{Preliminaries}
We fix some conventions and recall a fact about cofibrations.

\begin{defn} Let $A\subset B$ and $X\subset Y$.  A map $f\colon B\rightarrow Y$ 
is a \emph{relative map} if $f(A)\subset X$.  
\end{defn}
We will write this \[f\colon 
(B,A)\rightarrow (Y,X).\]  
A homotopy 
$H\colon B\times I\rightarrow Y$ is a \emph{relative homotopy} if $H|_{A\times I}$
factors through the inclusion $X\subset Y$.

\begin{defn}\cite[3.1]{Zhao01}
A relative map $f\colon (B,A)\rightarrow (B,A)$ is \emph{taut} if there is a
neighborhood $N(A)$ of $A$ in $B$ such that $f(N(A))\subset A$.
\end{defn}

We will use this condition to isolate the fixed points of $A$ from those of $B\setminus A$.

\begin{lem}\cite[3.2]{Zhao01} If $A\subset B$ is a cofibration then any relative 
map $f\colon (B,A)\rightarrow (B,A)$ is relatively homotopic to a taut map.
\end{lem}

We will assume $A\subset B$ is a cofibration and all relative maps \[f\colon 
(B,A)\rightarrow (B,A).\] are taut.
If a relative map is not taut it is implicitly replaced by a relatively homotopic map that is 
taut.  Since all invariants defined here are invariants of the relative homotopy class, 
the choice of replacement does not matter.

%% file: relative_fp_1_2.tex
\section{Duality and trace in symmetric monoidal categories}\label{symmetricmonoidal}

Duality and trace in symmetric monoidal categories is a generalization of
the trace in linear algebra that retains many important properties.  The trace
in a symmetric monoidal category
satisfies a generalization of invariance of basis and has nice functorial properties.  The Lefschetz fixed
point theorem is one application of the  functoriality properties of the trace.
This is a very brief summary of  \cite{DP}.  For more details see \cite{DP}, \cite[Chapter III]{LMS},
or \cite{shadows}.

Let $\sC$ be a symmetric monoidal category with monoidal product $\otimes$, unit $S$, and symmetry
isomorphism \[\gamma\colon X\otimes Y\rightarrow Y\otimes X.\]
\begin{defn} An object $A$ in $\sC$ is {\em dualizable} with dual $B$ if there are maps 
\[\eta\colon S\rightarrow A\otimes B\] and \[\epsilon\colon B\otimes A\rightarrow S\] such that the 
composites \[\xymatrix{A\cong S\otimes A\ar[r]^-{\eta\otimes\id}&A\otimes B\otimes A
\ar[r]^-{\id \otimes \epsilon}&A\otimes S\cong A}\] and 
\[\xymatrix{B\cong B\otimes S\ar[r]^-{\id\otimes \eta}&B\otimes A\otimes B
\ar[r]^-{\epsilon \otimes \id}&S\otimes B\cong B}\] are the identity maps of $A$ and $B$ 
respectively.
\end{defn}

The most familiar example of a symmetric monoidal category is the category of modules
over a commutative ring $R$.  The tensor product is the monoidal product.  
If $M$ is a finitely generated projective $R$-module,
$M$ is dualizable and the dual of $M$ is $\Hom _R(M,R)$.  The evaluation map
\[\epsilon\colon \Hom_R(M,R)\otimes_R M\rightarrow R\]  is defined by  $\epsilon(\phi,m)=\phi(m)$.  
Since $M$ is finitely generated and projective the dual basis theorem implies 
there is a `basis' $\{m_1,m_2,\ldots,m_n\}$ with dual `basis' $\{m_1',
m_2',\ldots,m_n'\}$.  The coevaluation is given by linearly extending 
\[\eta(1)=\sum m_i\otimes m_i'.\]

The category of chain complexes of modules over a commutative ring $R$ is 
also symmetric monoidal.  The dualizable objects are the chain complexes
that are projective in each degree and finitely generated.  

\begin{defn}
If $A$ is dualizable and $f\colon A\rightarrow A$ is an endomorphism in $\sC$, 
the \emph{trace} of $f$, $\tr(f)$, is the 
composite \[\xymatrix{ S\ar[r]^-\eta&A\otimes B\ar[r]^{f\otimes \id}&A\otimes 
B\ar[r]^\gamma&B\otimes A\ar[r]^-\epsilon&S}.\]
\end{defn}

The trace of an endomorphism in the symmetric monoidal category of vector spaces 
over a field is
the sum of the diagonal elements in a matrix representation. 
The trace of an endomorphism in the category of chain complexes of modules
over a commutative ring $R$ is called the \emph{Lefschetz number}.

\begin{prop}\label{symtracefunctor} Let $F\colon \sC\rightarrow \sD$ be
 a symmetric monoidal functor, $A$ be an object of  $\sC$ with dual $B$,
and \[F(A)\otimes F(B)\rightarrow F(A\otimes B)\] and 
\[S_{\sD}\rightarrow F(S_\sC)\] be isomorphisms.  
Then $F(A)$ is dualizable with dual $F(B)$.  
If $f\colon A\rightarrow
A$ is an endomorphism in $\sC$, $F(\tr(f))=\tr(F(f))$.
\end{prop}

The  stable homotopy category is a symmetric monoidal category.
There is also a way to express duality for spaces without using spectra.

\begin{defn} A based space $X$ is $n$-\emph{dualizable} if there is a based 
space $Y$ and continuous maps 
$\eta\colon S^n\rightarrow X\wedge Y$ and 
$\epsilon\colon Y\wedge X\rightarrow S^n$ such that the diagrams 
\[\xymatrix{S^n\wedge X\ar[r]^-{\eta\wedge \id}\ar[dr]_\gamma&
X\wedge Y\wedge X\ar[d]^-{\id\wedge \epsilon}&Y\wedge S^n\ar[r]^-{\id\wedge \eta}\ar[dr]_{(
\sigma\wedge 1)\gamma} &Y\wedge X\wedge Y\ar[d]^{\epsilon\wedge \id}\\
&X\wedge S^n&&S^n\wedge Y
}\] commute up to stable homotopy.
\end{defn}

The map $\sigma\colon S^n\rightarrow S^n$ is defined by  $\sigma(v)=-v$.

\begin{prop}\cite[III.4.1, III.5.1]{LMS}\cite[18.6.5]{MS}\label{topduals}
\begin{enumerate}
\item If $M$ is a closed smooth manifold that embeds in $\mathbb{R}^m$, then 
$M_+$ is dualizable with dual $T\nu$, the Thom space of the normal bundle of 
the embedding of $M$ in $\mathbb{R}^m$.
\item If $L$ is a closed submanifold of a closed smooth manifold $M$ that embeds
in $\mathbb{R}^m$, then $M\cup CL$ is dualizable with dual 
$T{\nu_M}\cup CT{\nu_L}$.
\item If $B$ is a compact ENR that embeds in $\mathbb{R}^n$, $B_+$ is 
dualizable with dual the cone on the inclusion $\mathbb{R}^n\setminus B
\rightarrow \mathbb{R}^n$.
\item If $B$ is a compact ENR that embeds in $\mathbb{R}^n$ and 
$A$ is a sub ENR of $B$, then $B\cup CA$ is dualizable with dual $(\mathbb{R}^n\setminus A)
\cup C(\mathbb{R}^n\setminus B)$.
\end{enumerate}
\end{prop}

Here $C$ denotes the cone.  If $A\subset B$ then $B\cup CA$ is the mapping 
cone on the inclusion $A\rightarrow B$.  The base point of $B\cup CA$ is the 
cone point.

The trace of an endomorphism of spaces regarded as a map in the stable homotopy 
category is called the {\em fixed point index}.  The index is the stable homotopy class
of a map \[S^n\rightarrow S^n\] and so is an element of 
the $0^{th}$ stable homotopy group of $S^0$, $\pi_0^{s}$.  For other
descriptions of the fixed point index see \cite{Brownbook, Doldbook}.

The index of the identity map of a space $X$ is called the {\em Euler characteristic}
of $X$ and it is denoted 
$\chi(X)$.  This is consistent with the classical definition of the Euler characteristic.

Since the rational homology functor is strong symmetric monoidal, 
\autoref{symtracefunctor} implies
that the fixed point index of a map $f$ is equal to the Lefschetz number of $H_*(f)$.
Since the fixed point index of a map with no fixed points is zero, this implies the 
Lefschetz fixed point theorem.

\section{Duality and trace for bicategories with shadows}\label{dualbicat}

Unfortunately, the Reidemeister trace cannot be defined as a trace 
in a symmetric monoidal category.  It can be defined using the more general 
trace in a bicategory with shadows.  Duality and trace in a 
bicategory are very similar to duality and trace in a symmetric monoidal
category but are more flexible.  
This is a brief summary of the relevant sections of \cite{MS} and \cite{thesis}.  For more 
details see \cite[Chapter 16]{MS}, \cite{thesis}, or \cite{shadows}.

\begin{defn}\cite[1.0]{Leinster} A \emph{bicategory}
$\sB$ consists of
\begin{enumerate}\item A collection $\ob\sB$.
\item Categories $\sB(A,B)$ for each $A,B\in \mathrm{ob}\sB$.
\item Functors \[\odot \colon \sB(B,C)\times \sB(A,B)\rightarrow
\sB(A,C)\]
\[U_A\colon \ast \rightarrow \sB(A,A)\]
for $A$, $B$ and $C$ in $\mathrm{ob}\sB$.
\end{enumerate}
Here $\ast$ denotes the category with one object and one morphism.
The functors $\odot$ are  required to satisfy unit and
associativity conditions  up to  natural isomorphism 2-cells.
\end{defn}
The elements of $\ob\sB$ are called 0-cells.  The objects of $\sB(A,B)$
are called 1-cells.  The morphisms of $\sB(A,B)$ are called 2-cells.

The most familiar example of a bicategory is the bicategory $\Mod$ 
with 0-cells rings, 1-cells bimodules,
and 2-cells homomorphisms.  The bicategory composition is tensor product.

\begin{defn}\cite[16.4.1]{MS} A 1-cell $X\in \sB(B,A)$ is \emph{right dualizable} with 
dual $Y\in \sB(A,B)$ if there are 2-cells \[\eta\colon U_A\rightarrow X\odot Y\hspace{1cm}\epsilon\colon Y\odot X
\rightarrow U_B\] such that 
\[\xymatrix@C=10pt@R=10pt
{Y\ar[r]^-\sim \ar[dddrrr]^\id
&Y\odot U_A\ar[rr]^-{\id \odot \eta}
&&Y\odot X\odot Y\ar[dd]^{\epsilon\odot \id}
&X\ar[r]^-\sim\ar[dddrrr]^\id
&U_A\odot X\ar[rr]^-{\eta\odot \id}
&&X\odot Y\odot X\ar[dd]^{\id \odot \epsilon}
\\
\\
&&&U_B\odot Y\ar[d]^\wr
&&&&X\odot U_B\ar[d]^\wr
\\&&& Y 
&&&&X}\]
commute.
\end{defn}
The map $\eta$ is called the {\em coevaluation} and $\epsilon$ is called the {\em evaluation}.

If $R$ is a (not necessarily commutative) ring 
and $M$ is a finitely generated projective right $R$-module then $M$ is a right 
dualizable 1-cell in  
${\Mod}$ with dual $\Hom_R(M,R)$.  The evaluation map 
\[\epsilon\colon\Hom_R(M,R)\otimes_\mathbb{Z}M\rightarrow R\] is defined by 
$\epsilon(\phi,m)=\phi(m)$.  This is a map of $R$-$R$-bimodules. 
Since $M$ is finitely generated and projective there are elements $\{m_1, m_2, \ldots
, m_n\}$ and dual elements $\{m_1', m_2', \ldots, m_n'\}$ of $\Hom_R(M,R)$ so that the 
coevaluation map \[\eta\colon\mathbb{Z}\rightarrow M\otimes_R \Hom_R(M,R)\]
is defined by linearly extending $\eta(1)=\sum m_i\otimes m_i'$.  This is a map of abelian groups. 

Unlike the symmetric monoidal case, we need more structure before we can define trace.
The additional structure is a shadow.

\begin{defn}\cite[4.4.1]{thesis}
A \emph{shadow} for $\sB$ is a functor \[\sh{-}\colon \coprod_{A\in\ob \sB}
\sB(A,A)\rightarrow \sT\] to a 
category $\sT$ and a natural isomorphism \[\theta\colon \sh{X\odot Y}\cong \sh{Y\odot X}\] for every pair of 
1-cells $X\in \sB(A,B)$ and $Y\in \sB(B,A)$ such that 
\[\xymatrix{\sh{(X\odot Y)\odot Z}\ar[r]^\theta\ar[d]&\sh{Z\odot (X\odot Y)}
\ar[r]&\sh{(Z\odot X)\odot Y}\\
\sh{X\odot(Y\odot Z)}\ar[r]_\theta&\sh{(Y\odot Z)\odot X}\ar[r]&
\sh{Y\odot (Z\odot X)}\ar[u]_\theta
}\]
\[\xymatrix{\sh{Z\odot U_A}\ar[r]^\theta\ar[dr]&\sh{U_A\odot Z}\ar[d]
\ar[r]^\theta&\sh{Z\odot U_A}\ar[dl]\\&\sh{Z}}\]
commute. 
\end{defn}

Let $P$ be an $R$-$R$-bimodule.  Let $N(P)$ be the subgroup of $P$ generated by 
elements of the form \[rp-pr\] for $p\in P$ and $r\in R$.  Then the shadow of $P$ 
is $P/N(P)$. 

\begin{defn}\cite[4.5.1]{thesis}
Let $X$ be a dualizable 1-cell and $f\colon Q\odot X\rightarrow X\odot P$ be a 2-cell in 
$\sB$.  The \emph{trace} of $f$ is the composite
\[\xymatrix{{\sh{Q}}\cong \sh{Q\odot U_A}\ar[r]^-{\id \odot \eta}&
\sh{Q\odot X\odot Y}\ar[d]^-{f\odot \id}\\
&\sh{X\odot P\odot Y}\ar[r]^\theta&
\sh{P\odot Y\odot X}\ar[r]^-{\id \odot \epsilon}&\sh{P\odot U_B}\cong\sh{P}.}\] 
\end{defn}

If $M$ is a finitely generated projective right $R$-module and $f\colon M\rightarrow M$ is a map 
of right $R$-modules the trace of $f$ is the trace defined by Stallings in \cite{Stallings}.

A functor of bicategories $F$ is a \emph{shadow functor}
 if there is a natural transformation 
\[\psi\colon \sh{F(-)}\rightarrow F(\sh{-})\] such that 
\[\xymatrix{\sh{FX\odot FY}\ar[r]^\theta\ar[d]&\sh{FY\odot F(X)}\ar[d]\\
\sh{F(X\odot Y)}\ar[d]^\psi&\sh{F(Y\odot X)}\ar[d]^\psi\\
F(\sh{X\odot Y})\ar[r]^\theta&F(\sh{Y\odot X})}\] commutes for all 1-cells $X$ and $Y$
where $X\odot Y$ and $Y\odot X$ are both defined. 

\begin{prop}\label{tracefunctor}\cite[4.5.7]{thesis} Let $X$ be a right dualizable 1-cell in $\sB$ with dual $Y$, 
\[f\colon Q\odot X\rightarrow X\odot P\] be a 2-cell in $\sB$ and $F\colon \sB\rightarrow \sB'$
be a shadow functor.  If $F(X)\odot F(Y)\rightarrow F(X\odot Y)$,
$F(X)\odot F(P)\rightarrow F(X\odot P)$,
and $U_{F(B)}\rightarrow F(U_B)$ are isomorphisms and  $\hat{f}$ is the composite
\[\xymatrix{FQ\odot FX\ar[r]^-{\phi}&F(Q\odot X)\ar[r]^{F(f)}&F(X\odot P)\ar[r]^{\phi^{-1}}&FX\odot FP}\]
then \[\xymatrix{\sh{FQ}\ar[r]^{\tr(\hat{f})}\ar[d]^\psi&\sh{FP}\ar[d]^\psi\\
F\sh{Q}\ar[r]^{F(\tr(f))}&F\sh{P}}\] commutes.
\end{prop}

We will use this proposition to compare different forms of the Lefschetz number and 
Reidemeister trace.

%% file: relative_fp_3.tex
\section{Some examples of bicategories with shadows}\label{bicatdist}
The classical descriptions of fixed point invariants require a choice of base point.
When working with a single space this isn't 
a problem.  With fiberwise spaces, equivariant spaces, or pairs of spaces,
choosing base points requires addition conditions on the space.  
We can avoid these choices by using groupoids rather than groups.

In this 
section we describe a generalization of the bicategory 
$\Mod$ that we will use to defined fixed point invariants without 
choosing base points.  In this bicategory  we replace rings by categories, modules by
 functors, and homomorphisms by natural transformations.

Let $\sV$ be a symmetric monoidal category with monoidal product $\otimes$ and 
unit $S$.  
\begin{defn} A category $\sA$ is {\em enriched} in $\sV$ if
for each $a,b\in \mathrm{ob}(\sA)$, $\sA(a,b)$ is an object of $\sV$ and the composition 
for $\sA$, 
\[\sA(b,c)\otimes \sA(a,b)\rightarrow \sA(a,c),\] is a morphism in $\sV$.  
\end{defn}
For pairs of enriched categories $\sA$ and $\sB$ define an enriched category $\sA\otimes \sB$
with objects  
pairs $(a,b)$ where $a\in\ob\sA$ and $b\in\ob\sB$.  
If $a,a'\in\ob\sA$ and $b,b'\in\ob\sB$, then \[(\sA\otimes \sB)((a,b),
(a',b'))=(\sA(a,a'))\otimes(\sB(b,b')).\]

\begin{defn} An \emph{enriched distributor} is a functor $\sX\colon \sA\otimes \sB^{\op}\rightarrow \sV$ 
such that the action of morphisms of $\sA$ and 
$\sB$ on $\sX$ are maps in $\sV$.  
\end{defn}

This type of functor is also called an $\sA$-$\sB$-bimodule.  If $F\colon \sA\rightarrow \sC$ is an 
enriched functor and $\sY\colon \sC\otimes \sB^\op\rightarrow \sV$ is a distributor
define a new distributor $\sY^F\colon \sA\otimes \sB^\op\rightarrow \sV$ by 
$\sY^F(a,b)=\sY(F(a),b)$.

\begin{defn}An \emph{enriched natural transformation} $\eta\colon \sX\rightarrow \sY$ 
is a natural transformation
where the maps \[\eta_{a,b}\colon \sX(a,b)\rightarrow \sY(a,b)\] are maps in $\sV$ for
all $a\in \ob\sA$ and $b\in\ob\sB$. 
\end{defn}

Enriched categories are the 0-cells of a bicategory we denote by $\sE_\sV$.  The 1-cells
are the distributors enriched in $\sV$.  The 2-cells are enriched natural transformations.   
If $\sX\colon \sA\otimes \sB^{\op}\rightarrow \sV$ and $\sY\colon \sB\otimes \sC^{\op}\rightarrow \sV$
are two distributors, $\sX\odot \sY$ is a distributor $\sA\otimes \sC^{\op}\rightarrow \sV$.
For $a\in \ob(\sA)$ and $c\in \ob(\sC)$, $\sX\odot \sY(a,c)$ is the coequalizer of the
actions of $\sB$ on $\sX$ and $\sY$,
\[\xymatrix{\displaystyle{\coprod_{b,b'\in \ob\sB}\sX(a,b)\otimes \sB(b',b)\otimes \sY(b',c)}
\ar@<.5ex>[r]\ar@<-.5ex>[r]&
\displaystyle{\coprod_{b\in \ob\sB}\sX(a,b)\otimes \sY(b,c)}}.\]

If $\sZ\colon \sA\odot \sA^\op\rightarrow \sV$ is an enriched functor, the
\emph{shadow} of $\sZ$, $\sh{\sZ}$, is the coequalizer of the two actions 
of $\sA$ on $\sZ$,
\[\xymatrix{\displaystyle{\coprod_{a,a'\in \mathrm{ob}(\sA)}} \sA(a,a')\otimes \sZ(a,a')
\ar@<.5ex>[r]\ar@<1.5ex>[r]&\displaystyle{\coprod_{a\in \mathrm{ob}(\sA)}} \sZ(a,a)}.\]

In \cite[Chapter 9]{thesis} we observed that if $\sA$ is a connected groupoid, a distributor 
\[\sX\colon \sA\rightarrow \sV\] is dualizable if and only if $\sX(a)$ is dualizable over $\sA(a,a)$ for any 
$a\in \ob \sA$.  The categories we will use here and in \cite{equiv} to define relative
and equivariant fixed point invariants are not usually groupoids, but we can extend the
results from \cite{thesis} to describe these examples.  

\begin{defn}\cite[II.9.2]{Luck} A category $\sA$ is an \emph{EI-category} 
if all endomorphisms are isomorphisms.
\end{defn}

In an EI-category there is a partial order on the set of objects: $x<y$  if $\sA(x,y)\neq
\emptyset$.  

Let $\Ch_R$ be the symmetric monoidal category of chain complexes of modules
over a commutative ring $R$ and chain maps.   Let $\sA$ be a category enriched
in the category of modules over $R$. This can be regarded as a category enriched in chain complexes
concentrated in degree zero.

\begin{defn}\label{trivialalg}A functor $\sX \colon \sA\rightarrow {\Ch_R}$ is {\em supported 
on isomorphisms}
if $\sX (f)$ is the zero map if $f$ is not an isomorphism.
\end{defn}

If $\sX$ is supported on isomorphisms it only `sees' a disjoint collection of groupoids 
rather than the entire category $\sA$. 

Let $B(\sA)$ be a choice of representative for each isomorphism 
class of objects in $\sA$.

\begin{lem}\label{specialtensor}  If $\sX \colon \sA^{op}\rightarrow {\Ch_R}$ and 
$\sY\colon \sA\rightarrow {\Ch_R}$ are supported on isomorphisms then 
\[\sX \odot \sY \cong \displaystyle{ \bigoplus_{c\in B(\sA) }} \sX (c)\otimes_{\protect\sA(c,c)}\sY(c).\] 
\end{lem}

The proof of this lemma is in \autoref{extracat}.  The idea of the proof is to 
use \autoref{trivialalg} to 
show that \[\displaystyle{ \bigoplus_{c\in B(\sA)}} \sX (c)\otimes_{\protect\sA(c,c)}\sY(c)\]
satisfies the universal property that defines $\sX \odot \sY$.   

\begin{lem}\label{explicitdual} Let $\sX $ and $\sY$ satisfy the conditions of 
\autoref{specialtensor}.
If $\sX (c)$ is dualizable as an $\sA(c,c)$-module
with dual $\sY(c)$ for each $c\in B(\sA)$ then $\sX $ is dualizable with dual $\sY$.
\end{lem}

The idea of this proof is to use  \autoref{specialtensor} and the coevaluation
and evaluation maps for each $\sX (c)$ to define coevaluation and evaluation maps for $\sX $.  
This proof can also be found in \autoref{extracat}.

Another choice for $\sV$ is the symmetric monoidal category of pointed topological
spaces, $\Top_*$.   The bicategory $\sE_{\Top_*}$ has 0-cells categories enriched
in based spaces and 1-cells distributors enriched in based spaces. The 2-cells in 
$\sE_{\Top_*}$ are natural transformations enriched in $\Top_*$.  

If $\sX^\op \colon \sA\rightarrow {\Top_*}$ and $\sY\colon \sA\rightarrow {\Top_*}$ are 
enriched functors $\sX \odot \sY$ is the bar resolution $B(\sX ,\sA,\sY)$.  This is the 
geometric realization of the simplicial space with $n$-simplices 
\[\coprod_{a_0,a_1,\ldots ,a_n\in \ob \sA} \sX (a_0)\wedge \sA(a_1,a_0)_+\ldots
\wedge \sA(a_n,a_{n-1})_+\wedge \sY(a_n).\]
The definition of the shadow is similar.
If $\sZ\colon \sA\otimes \sA^\op\rightarrow {\Top_*}$ is a enriched functor,
the shadow of $\sZ$, $\sh{\sZ}$, is the cyclic bar resolution $C(\sA,\sZ)$.
This is the geometric realization of the simplicial space with $n$-simplices 
\[\coprod_{a_0,a_1,\ldots ,a_n\in \ob \sA}  \sZ (a_n,a_0)\wedge \sA(a_1,a_0)_+
\ldots \wedge \sA(a_n,a_{n-1})_+.\]

Let $\sA$ be a category enriched in based spaces.  
Let $U_{\sA}\colon \sA\otimes\sA^\op\rightarrow \Top_*$ be defined by 
$U_{\sA}(a,a')=\sA(a',a)$.  Composition in $\sA$ defines the action of 
$\sA$ and $\sA^\op$ on $U_{\sA}$.

\begin{defn} An enriched functor $\sX \colon \sA\rightarrow {\Top}_*$ is $n$-\emph{dualizable} if there is a
functor $\sY\colon \sA^{\op}\rightarrow {\Top}_*$, a map 
$\eta\colon S^n\rightarrow B(\sX ,\sA,\sY)$, and an $\sA$-$\sA$-equivariant map
$\epsilon\colon \sY\wedge \sX \rightarrow S^n\wedge U_\sA$
such that the usual triangle diagrams commute up to $\sA$-equivariant 
homotopy.
\end{defn}

We will use the ideas of \autoref{explicitdual} to produced dual pairs in 
this bicategory, but we will not prove a general characterization.

\begin{defn} If $\sX \colon \sA\rightarrow {\Top}_*$ is dualizable, 
$\sP\colon \sA\otimes \sA^{\op}\rightarrow \Top_* $ is an enriched functor 
and $f\colon \sX\rightarrow \sX\odot \sP$ is an enriched natural transformation  
the {\em trace} of $f$ is the stable homotopy class of the composite
\[\xymatrix{S^n\ar[r]^-{\eta}&
\sh{\sX\odot \sY}\ar[r]^-{f\odot \id}
&\sh{\sX\odot \sP\odot \sY}\ar[r]^\theta&
\sh{\sP\odot \sY\odot \sX}\ar[r]^-{\id \odot \epsilon}&S^n\wedge \sh{\sP}}.\] 
\end{defn}

%% file: relative_fp_4.tex
\section{The relative Lefschetz number}\label{relidxsec}
The global and geometric relative Lefschetz numbers can both be  
described using a classical approach, but we will describe
them using duality and trace in a bicategory.  The formal structure
gives a different perspective on these invariants and 
is a starting point for the more complicated invariants
we will consider in the later sections.

\begin{defn}\label{comcat} The \emph{relative component category} 
$\rcomcat{B}{A}$ of a pair
$(B,A)$ has objects the points of $B$.
The morphisms of $\rcomcat{B}{A}$ are 
\[\rcomcat{B}{A}(x,y)=\left\lbrace \begin{array}{ll}
\ast&\mathrm{if}\,x\in B\setminus A\,\mathrm{and}\, [x]=[y]\in\pi_0(B)\\
\varnothing&\mathrm{if}\,x\in B\setminus A\,\mathrm{and}\, [x]\not=[y]\in\pi_0(B)\\
\ast&\mathrm{if}\,x,y\in A\,\mathrm{and}\, [x]=[y]\in\pi_0(A)\\
\varnothing&\mathrm{if}\,x,y\in A\,\mathrm{and}\, [x]\not=[y]\in\pi_0(A)\\
\varnothing&\mathrm{if}\, x\in A, \, y\not\in A
\end{array}\right. 
\]
The composition is defined by the rules
\[\begin{array}{ll}
\ast\circ\ast=\ast &\varnothing\circ \varnothing=\varnothing\\
\varnothing\circ\ast=\varnothing&\ast\circ\varnothing=\varnothing.
\end{array}\]
\end{defn}

For most pairs of spaces $A\subset B$ this category is an EI-category but not a groupoid.
For example, if $x\in B\setminus A$, $y\in A$, and $[x]=[y]\in \pi_0(B)$ then  $\rcomcat{B}{A}(x,y)=\ast$
and  $\rcomcat{B}{A}(y,x)=\varnothing$.  
The relative component category is similar to the
equivariant component category in \cite[I.10.3]{tomDieck}.

If $A$ and $B$ are connected and 
$B\setminus A$ is nonempty this category has two isomorphism classes of objects.

If $x\in A$, let $A(x)$ be the component of $A$ that contains $x$.  If $y\in B$, let 
$B(y)$ be the component of $B$ that contains $y$.

\begin{defn} The \emph{relative component space}, $\rcomsp{B}{A}$, of the 
pair $(B,A)$ is the functor \[\rcomcat{B}{A}^\op\rightarrow \Top_*\] 
defined by \[\rcomsp{B}{A}(x)=\left\lbrace \begin{array}{ll}
A(x)_+&\mathrm{if}\,x\in A\\
B(x)/(A\cap B(x))&\mathrm{if}\,x\not \in A
\end{array}\right. 
\] The morphisms $A(x)\rightarrow A(x)$ and $B(x)/ (A\cap B(x))\rightarrow
B(x) / (A\cap B(x))$ are the identity maps.  The map $A(x)\rightarrow B(x)/ (A\cap B(x))$
is the inclusion of $A(x)$ as the base point.
\end{defn}

Recall that $C(A\cap B(x))$ is the cone on $A\cap B(x)$.  The base point is the cone
point.  Since $A\subset B$ is assumed to be a cofibration $B(x)\cup C(A\cap B(x))$ is 
homotopy equivalent to $B(x)/(A\cap B(x))$.

\begin{lem}\label{comspdual} If $A$ and $B$ are both compact ENR's or closed  smooth 
manifolds then ${\rcomsp{B}{A}}$ is dualizable.
\end{lem}

\begin{rmk} Starting with the proof of this theorem we will focus on the case of closed
smooth manifolds.  The results in this section and Sections \ref{geosec} and \ref{algsec} 
have versions for compact ENR's as well.  
The statements and proofs for compact ENR's are very similar to those for closed 
smooth manifolds.

Some of the results in 
\autoref{converse} have only been shown for manifolds.  
\end{rmk}

\begin{proof}
Define a functor $D({\rcomsp{B}{A}})\colon \rcomcat{B}{A}\rightarrow
\Top_*$ by 
 \[D(\rcomsp{B}{A})(x)=\left\lbrace \begin{array}{ll}
D(A(x)_+)&\mathrm{if}\,x\in A\\
D(B(x)\cup C (A\cap B(x)))&\mathrm{if}\,x\not \in A
\end{array}\right. 
\]  where $D(A(x)_+)$ and $D(B(x)\cup C (A\cap B(x)))$ denote the 
duals of $A(x)_+$ and $B(x)\cup C (A\cap B(x))$ with respect to an embedding 
of $B$ in $\mathbb{R}^n$ as described in 
\autoref{topduals}.  The morphisms are the identity maps or the inclusion.

To simplify notation, consider the case where $A$ and $B$ are both connected.
The general case is similar.  
The evaluation for this dual pair is a natural transformation 
\[\epsilon\colon D({\rcomsp{B}{A}})\odot {\rcomsp{B}{A}}
\rightarrow S^n\wedge (\rcomcat{B}{A})_+.\]
Let $x\in A$ and $y\in B\setminus A$ represent the isomorphism classes of
objects of $\rcomcat{B}{A}$.  Then $\epsilon$ consists of 
four maps 
\[ D(B\cup C (A)) \wedge \left(B/A\right)
\rightarrow S^n\]
\[D(A_+)\wedge \left(B/A\right)\rightarrow S^n\]
\[D(B\cup C (A))\wedge A_+\rightarrow *\]
\[D(A_+)\wedge A_+\rightarrow S^n\]
By naturality, the second map must be the constant map to a point.  Since 
$A_+$ and $B\cup C (A)$ are both dualizable the evaluations for 
these dual pairs are the first and fourth maps.

Note that $B(\rcomsp{B}{A}, \rcomcat{B}{A}, D(\rcomsp{B}{A}))$
is equivalent to  \[(A_+\wedge D(A_+))\vee  [(B/A)
\wedge D(B\cup CA)].\] 
The dualizability of $A_+$ and $B\cup C (A)$ provide coevaluation
maps \[\eta_A\colon S^n\rightarrow A_+\wedge D(A_+) \]
\[\eta_{B/A}\colon S^n\rightarrow   B/A\wedge D(B\cup C (A)).\]
The coevaluation for this dual pair is the composite
\[\xymatrix{S^n\ar[r]^-\triangle&S^n\vee S^n\ar[d]^-{\eta_A\vee\eta_{B/A}}\\
&(A_+\wedge D(A_+))\vee  [(B/A)
\wedge D(B\cup CA_+)].}\]

Verifying that these maps describe a dual pair can be checked on $x$ and $y$ separately.
The conditions reduce to conditions checked for  \autoref{topduals}.
\end{proof} 

Let $\rcomcatf{B}{A}$ be the functor \[\rcomcat{B}{A}\times
\rcomcat{B}{A}^\op\rightarrow \Top_*\] defined by 
$\rcomcatf{B}{A}(x,y)=\rcomcat{B}{A}(f(y),x)_+$.  The left action is 
composition.  The right action is given by applying $f$ and then the composition.

A relative map $f\colon (B,A)\rightarrow (B,A)$ induces a natural
transformation \[\rcomspm{f}\colon \rcomsp{B}{A}\rightarrow \rcomsp{B}{A}\odot 
\rcomcatf{B}{A}.\]
Since ${\rcomsp{B}{A}}$ is dualizable, the trace of $\rcomspm{f}$ is defined.

\begin{defn} The \emph{relative geometric  Lefschetz number} of $f$, $\relidx{f}{B}{A}$, 
is the trace of $\rcomspm{f}$.
\end{defn}

The relative geometric Lefschetz number is the stable homotopy class of a map 
\[S^0\rightarrow \sh{\rcomcatf{B}{A}}\] and so it is an element of the 
$0^{th}$ stable homotopy group of $\sh{\rcomcatf{B}{A}}$.  This group
is denoted $\pi_0^s(\sh{\rcomcatf{B}{A}})$.  It is the free abelian group
on the set $\sh{\rcomcatf{B}{A}}$.  Since the relative geometric Lefschetz number is  
defined to be the trace of $\rcomspm{f}$ it is an invariant of the relative homotopy 
class of $f$.

If $\comcat{A}$ is the set of component of $A$, $\sh{\comcatf{A}}$ 
is \[\left\lbrace [x]\in \pi_0(A)|[f|_A(x)]=[x]\right\rbrace.\]

If $X$ is a closed smooth manifold, $f\colon X\rightarrow X$ is a continuous map and 
$F$ is an isolated subset of the set of fixed points of $f$ let $i(F,f)$ be the sum of the 
fixed point indices of the fixed points in $F$.  See \cite{Doldbook} for the definition of the 
fixed point index of an isolated set of fixed points.

\begin{lem}\label{splitcomponents} \label{splitidx}
There is an isomorphism 
\[\sh{\rcomcatf{B}{A}}\cong \sh{\comcatf{A}} \amalg \sh{\comcatf{B}}\]
and the image 
of $\relidx{f}{B}{A}$ under this isomorphism 
is \[\sum_{x\in\sh{\comcatf{A}}}i(\mathrm{Fix}(f)\cap A(x), f)[x]
+\sum_{y\in\sh{\comcatf{B}}} i(\mathrm{Fix}(f)\cap (B(y)\setminus A), f))[y].\]
\end{lem}

Since $f$ is taut $i(\mathrm{Fix}(f)\cap A, f)=i(\mathrm{Fix}(f)\cap A,f|_A)$ and 
\[i(\mathrm{Fix}(f)\cap (B\setminus A), f)=i(\mathrm{Fix}(f),f)-i(\mathrm{Fix}(f)\cap A,
f|_A).\]   The first equality follows from commutativity of the index and 
the definition of a taut map. See \cite[3.5]{Zhao05} for a proof of the second equality.

\begin{proof} Assume $A$ and $B$ are connected and $A$ is a proper 
subset of $B$.  These assumptions restrict the number of 
components of $\sh{\rcomcatf{B}{A}}$.  The proof is similar for the general
case.

The set $\sh{\rcomcatf{B}{A}}$ is defined to be the coequalizer
\[\xymatrix{\protect\coprod_{x,y} {\rcomcat{B}{A}}(x,y)_+\wedge 
\protect{\rcomcatf{B}{A}}(y,x)\ar@<.5ex>[r]\ar@<-.5ex>[r]&
\protect\coprod_x {\rcomcatf{B}{A}}(x,x)\ar[r]&\sh{\rcomcatf{B}{A}}}.\]
Since $\rcomcat{B}{A}(x,y)$ is empty if $x\in A$ and $y\not\in A$, 
this coequalizer splits as two coequalizers.  One is over pairs $(x,y)$
where $x,y\in A$ and the other is over pairs $(x,y)$ where $x,y \not\in A$.
Each of these coequalizers consists of a single element.

This isomorphism  is compatible with 
\autoref{comspdual} so the image of $\relidx{f}{B}{A}$ under
the projection to the first summand is the trace of $f$ restricted to $A$.  
As observed after \autoref{topduals} this is the 
fixed point index of $f|_A$.

The image of $\relidx{f}{B}{A}$ under the projection to the second summand
is the trace of $f/A\colon B/A\rightarrow B/A$. 
The fixed points of $f/A$ are the fixed points of $f|_{B\setminus A}$ and the 
point that represents $A$.  The point that represents $A$ is the base point and so 
its index does not contribute to the trace of $f/A$, see \cite[III.8.5]{LMS}.
\end{proof}

The second component of $\relidx{f}{B}{A}$ is the index defined in 
\cite[1.1]{Jezierski95}.

\begin{eg} Let $J$ be a nonempty, proper, connected subinterval of $S^1$.  Let 
$f\colon (S^1,J)\rightarrow (S^1,J)$ be the identity map.  Then $\relidx{f}{S^1}{J}
=(1,-1)$.
\end{eg}

\begin{cor}\label{idxzero} If $f\colon (B,A)\rightarrow (B,A)$ has no fixed points then 
$\relidx{f}{B}{A}=0$.
\end{cor}

\begin{proof}  Since $f$ has no fixed points $i(\mathrm{Fix}(f),f)=0$.  To compute 
the relative geometric Lefschetz number 
 of $f$ we replace $f$ by a relatively homotopic map $g$ that is taut.  The map $g$
can be chosen so that $f|_A=g|_A$.  Then $i(\mathrm{Fix}(g),g)=0$ and 
$i(\mathrm{Fix}(g)\cap A, g|_A)=i(\mathrm{Fix}(f)\cap A, f|_A)=0$.

Since $f$ has no fixed points and $f|_A=g|_A$, $i(\mathrm{Fix}(g)\cap A, g)=0$.  
Since $g$ is taut 	\[i(\mathrm{Fix}(g)\cap (B\setminus A), g)=i(\mathrm{Fix}(g),g)-
i(\mathrm{Fix}(g)\cap A,g|_A)=0.\] 
\end{proof}

Let $\bZ\rcomcat{B}{A}$ be the category with the same objects as $\rcomcat{B}{A}$.
For objects $x$ and $y$ of $\rcomcat{B}{A}$ \[\bZ\rcomcat{B}{A}(x,y)\] is the free abelian group
on $\rcomcat{B}{A}(x,y)$.
Composing $\rcomsp{B}{A}$ with the rational homology functor
defines a functor \[H_*(\rcomsp{B}{A})\colon \bZ\rcomcat{B}{A}\rightarrow 
\Ch_\bQ.\]  

\begin{prop}\label{algspacedual} If $A\subset B$ are  closed smooth manifolds, 
then $H_*(\rcomsp{B}{A})$ is dualizable.
\end{prop}

\begin{proof}
There are two ways to prove this proposition.  First, the rational homology functor 
is strong symmetric monoidal, so this follows from \autoref{tracefunctor}.

We can also show $H_*(\rcomsp{B}{A})$ is dualizable directly by describing
the coevaluation and evaluation.  The functor $H_*(\rcomsp{B}{A})$ is 
supported on isomorphisms and so it is enough to construct a coevaluation 
and evaluation for the chain complexes of vector spaces $H_*(A)$ and $H_*(B,A)$.  These 
are both finite dimensional, and so they are both dualizable with 
duals as in \autoref{symmetricmonoidal}.
\end{proof}

A relative map $f\colon (B,A)\rightarrow (B,A)$ induces a map
\[H_*(f)\colon H_*(\rcomsp{B}{A}) \rightarrow H_*(\rcomsp{B}{A})\odot \bZ\rcomcatf{B}{A}\]
by applying the rational homology functor to $\rcomspm{f}$.

\begin{defn} The \emph{relative global Lefschetz number} of $f$, $\rlnumb{f}{B}{A}$,
is the trace of $H_*(f)$.
\end{defn}

\begin{lem}  The image of $\rlnumb{f}{B}{A}$ under the isomorphism in \autoref{splitcomponents} is \[\sum_{x\in\sh{\comcatf{A}}}\alnumb{f}{A(x)}[x]+\sum_{y\in\sh{\comcatf{B}}}
\alnumb{f}{B(y)/(A\cap B(y))}[y].\]
\end{lem}

Here $\alnumb{f}{A(x)}$ and $\alnumb{f}{B(y)/(A\cap B(y))}$ denote the traces of 
the maps induced by $f$ on $H_*(A(x))$ and $H_*(B(y)/(A\cap B(y)))$. 

\begin{proof}
Using \autoref{algspacedual} this proof is similar to the proof of \autoref{splitidx}.
\end{proof}

The invariant $\alnumb{f}{B/A}$ is the relative Lefschetz number of 
\cite{Bowszyc}.

\begin{prop} In $\mathbb{Z}\sh{\rcomcatf{B}{A}}$, $\rlnumb{f}{B}{A}
=\relidx{f}{B}{A}$.
\end{prop}

\begin{proof}
This proposition follows from  \autoref{tracefunctor} and the observation
that the rational homology functor is strong symmetric monoidal.
\end{proof}

The relative Lefschetz fixed point theorem follows from this proposition and 
\autoref{idxzero}.

\begin{thmA}[Relative Lefschetz fixed point theorem]
Let $A\subset B$ be closed smooth manifolds and $f\colon (B,A)\rightarrow
(B,A)$ be a relative map. If $f$ has no fixed points then $\rlnumb{f}{B}{A}
=0$.
\end{thmA}

Further, if $\rlnumb{f}{B}{A}\not=0$ all maps relatively homotopic 
to $f$ have a fixed point.

%% file: relative_fp_5.tex
\section{The geometric Reidemeister trace}\label{geosec}

To prove a converse to Theorem A it is necessary to introduce refinements of 
the invariants defined in the previous section.  The first of these invariants is the 
geometric Reidemeister trace.  This is a refinement of the geometric Lefschetz number 
and it will serve as a transition between the global Reidemeister trace
in \autoref{algsec} and the invariant defined in  \autoref{converse}.

As for the invariants in the previous section, it is possible to define the geometric 
Reidemeister trace using a generalization of the standard approach of fixed point indices and fixed point
classes.  Also as in the previous section, we do not use that approach here.  Instead
we use duality and trace in bicategories with shadows.  This perspective gives 
simple comparisons of different invariants and also unifies the descriptions of 
different forms of the Reidemeister trace with the Lefschetz number.

\begin{defn}The \emph{relative fundamental category}, $\rfuncat{B}{A}$,
of the pair $(B,A)$ has objects the points of $B$.  The morphisms $\rfuncat{B}{A}(x,y)$ are the homotopy classes
of paths from $x$ to $y$ in $A$ if $x\in A$ and homotopy classes of paths from $x$ to 
$y$ in $B$ if $x\in B\setminus A$.
\end{defn}

The relative fundamental category is a subcategory of the fundamental groupoid of $B$.
In most cases it is not a groupoid.   For example, if $A$ and $B$ are both path connected,
$x\in A$, and $y\in B\setminus A$ then 
$\rfuncat{B}{A}(x,y)$ is empty and $\rfuncat{B}{A}(y,x)$ is nonempty.  This category is an EI-category.
This category is similar to the equivariant fundamental category, see \cite[I.10.7]{tomDieck}.

For $x\in A$, let $\tilde{A}_x$  be the universal cover of $A$ based at $x$.  We think of 
points in $\tilde{A}_x$ as homotopy classes of paths in $A$ that start at $x$.  For $y\in 
B\setminus A$ let $\tilde{B}_y$ be the universal cover of $B$ based at $y$.  Let $p\colon 
\tilde{B}_y\rightarrow B$ be the quotient map and 
$\bar{A}_y=p^{-1}(A)\subset \tilde{B}_y$. 

\begin{defn} The \emph{relative universal cover} of the pair $(B,A)$ is
the functor \[\runicov{B}{A}\colon \rfuncat{B}{A}^{\op}\rightarrow \Top_*\] 
defined by 
\[\runicov{B}{A}(x)=\left\lbrace \begin{array}{ll}
(\tilde{A}_x)_+&\mathrm{if} \,x\in A
\\
{\tilde{B}_x/\bar{A}_x}&\mathrm{if} \,x\not\in A
\end{array}\right. \]
on objects and by composition of paths on morphisms.  
\end{defn}

\begin{lem} If $A\subset B$ is a cofibration $\tilde{B}_x/\bar{A}_x$ is 
$\pi_1(B)$-homotopy equivalent to $\tilde{B}_x\cup C\bar{A}_x$.
\end{lem} 

\begin{proof} 
There is a $\pi_1(B)$-equivariant map 
\[\phi\colon \tilde{B}_x\cup C\bar{A}_x\rightarrow \tilde{B}_x/\bar{A}_x\] defined
 by collapsing the cone to the base point.

Since $A\subset B$ is a cofibration there is a map \[u\colon 
B\rightarrow I\] such that $u^{-1}(0)=A$ and a homotopy \[h\colon B\times I
\rightarrow B\] such that $h(b,0)=b$ for all $b\in B$, $h(a,t)=a$ for all $a\in A$
and $t\in I$ and $h(b,1)\in A$ if $u(b)<1$.   The map 
\[\psi\colon \tilde{B}_x/\bar{A}_x\rightarrow \tilde{B}_x\cup C\bar{A}_x\] is 
defined by 
\[\psi(\gamma)=\left\lbrace \begin{array}{ll}
h(\gamma(1),t)|_{[0,2(1-u(\gamma(1)))]}\circ \gamma &\mathrm{if}\, \frac{1}{2}
\leq u(\gamma(1))\leq 1\\
(h(\gamma(1),t)\circ \gamma, 1-2u(\gamma(1)))&\mathrm{if}\, 0\leq u(\gamma(1))\leq 
\frac{1}{2}
\end{array}\right.\]
The map $\psi$ is $\pi_1(B)$ equivariant.  Up to homotopy it is an inverse for $\phi$.

\end{proof}

\begin{thm}\label{runivcovdual}
If $A\subset B$ are closed smooth manifolds the relative universal cover 
$\runicov{B}{A}$ is dualizable as a module over 
$\rfuncat{B}{A}$.  
\end{thm}

\begin{proof} The proof of this lemma is very similar to the proof of 
 \autoref{comspdual}.  We will define this dual pair by defining a dual 
pair for each isomorphism class of objects in $\rfuncat{B}{A}$.  

To simplify notation, consider the case where $A$ and $B$ are connected.
Let $S^{\nu_A}$ be the fiberwise one point compactification of the normal bundle of 
$A$.  This is a space over $A$ and has a section given by the points at infinity.
Let $D(\tilde{A}_+)$ be the space $(\tilde{A}\times_A S^{\nu_A})/\sim$ where
all points of the form $(\gamma, \infty_{\gamma(1)})$ are identified to a single
point.  This is the dual of $\tilde{A}_+$ as a distributor over $\pi_1(A)$, see \cite[5.3.3]{thesis}. 

Let  $C_B(S^{\nu_B},S^{\nu_A})$ be \[(B\times \{0\})\cup 
(S^{\nu_A}\times I) \cup (S^{\nu_B}\times \{1\}).\]
This is the fiberwise cone of the map $S^{\nu_A}\rightarrow S^{\nu_B}$  over $B$.
Let $D(\tilde{B}\cup C\bar{A})$ be the space 
\[(\tilde{B}\times_B C_B(S^{\nu_B},S^{\nu_A}))/\sim\] where
all points of the form $(\gamma, \gamma(1)\times \{1\})$ are identified to a single point.
This is the $\odot$ composition of the fiberwise spaces $(\tilde{B},p)_+$ and $C_B(S^{\nu_B},S^{\nu_A})$
defined in \cite[17.1.3]{MS}.
An argument like that in \cite[5.3.3]{thesis} for $\tilde{A}_+$ shows this is the dual
of $\tilde{B}\cup C\bar{A}$ as a distributor over $\pi_1(B)$.

The dual of $\runicov{B}{A}$, denoted $D(\runicov{B}{A})$, is 
\[D(\runicov{B}{A})(x)=\left\lbrace \begin{array}{ll}
D(\tilde{A}_+)&\mathrm{if} \,x\in A
\\
D(\tilde{B}\cup C\bar{A})&\mathrm{if} \,x\in B\setminus A
\end{array}\right. \] 
The action of the morphisms in $\rfuncat{B}{A}$ is by composition.

Using the assumption that $A$ and $B$ are connected there are two isomorphism classes of objects
in $\rfuncat{B}{A}$.  As in \autoref{comspdual}, there are four maps that 
define the natural transformation $\epsilon$.  Exactly as in that 
case there are only two that are nontrivial.  These maps are the evaluation 
maps for the dual pairs $(\tilde{A}_{+}, D(\tilde{A}_+))$
and $({\tilde{B}/\bar{A}}, D(\tilde{B}\cup C\bar{A}))$.

Also as in \autoref{comspdual}, $B(\runicov{B}{A}, \rfuncat{B}{A}, D(\runicov{B}{A}))$
is equivalent to \[\left( \tilde{A}_{+}\wedge_{\pi_1(A)} D(\tilde{A}_+)\right)  
\vee  \left( (\tilde{B}/\bar{A})\wedge_{\pi_1(B)} D(\tilde{B}\cup C\bar{A})\right). \]

The coevaluation map is the composite of the fold map
\[S^n\rightarrow S^n\vee S^n\] and the coevaluations for the pairs    
$(\tilde{A}_{+}, D(\tilde{A}_+))$
and $({\tilde{B}\cup C\bar{A}}, D(\tilde{B}\cup C\bar{A}))$,.

The required diagrams commute since the coevaluation and evaluation maps 
are defined using coevaluation and evaluation maps from the dual pairs
$(\tilde{A}_{+}, D(\tilde{A}_+))$
and $({\tilde{B}/ \bar{A}}, D(\tilde{B}\cup C\bar{A}))$.
\end{proof}

\begin{rmk}\label{explicitdescription}
We can also give more explicit descriptions of the coevaluation and evaluation maps
for the pairs $(\tilde{A}_{+}, D(\tilde{A}_+))$
and $({\tilde{B}/\bar{A}}, D(\tilde{B}\cup C\bar{A}))$.

The coevaluation for the pair $(\tilde{A}_{+}, D(\tilde{A}_+))$ is 
the composite
\[\xymatrix{S^n\ar[r]&T\nu_A\ar[r]& \tilde{A}_+\wedge_{\pi_1A}D(\tilde{A}_+)}\]
of the Pontryagin-Thom map for an embedding of $A$ in $S^n$ with 
the map $v\mapsto (\gamma,\gamma, v)$ where $\gamma$ is any element 
of $\tilde{A}$ that ends at the base of $v$.  

Since $A$ is locally contractible there is a neighborhood $U$ of the diagonal 
in $A\times A$ and a map 
\[H\colon V\rightarrow A^I\] that satisfies 
$H(x,x)(t)=x$, $H(x,y,0)=x$, and $H(x,y,1)=y$.  The evaluation for the pair 
$(\tilde{A}_{+}, D(\tilde{A}_+))$, 
\[D(\tilde{A}_+)\wedge \tilde{A}_+\rightarrow S^n\wedge \pi_1A_+\] 
is defined by 
\[(v,\gamma, \delta)=(\epsilon(v,\delta(1)), \gamma^{-1}H(\delta(1),\gamma(1))\delta)\]
where $\epsilon$ is the evaluation for the dual pair $(A_+,D(A_+))$.

The coevaluation and evaluation for the dual pair 
$({\tilde{B}/\bar{A}}, D(\tilde{B}\cup C\bar{A}))$ are similar.
\end{rmk}

A relative map $f\colon (B,A)\rightarrow (B,A)$ induces a map 
\[f_*\colon \runicov{B}{A}\rightarrow \runicov{B}{A}\odot \rfuncatf{B}{A}\]
where $\rfuncatf{B}{A}(x,y)=\rfuncat{B}{A}(f(y),x)_+$.  The left action is the 
usual left action.  The right action is given by applying $f$ and then composition.

\begin{defn}
The \emph{relative geometric Reidemeister trace} of $f\colon (B,A)\rightarrow (B,A)$,
$\rrgeo{f}{B}{A}$, is the trace of the map 
\[f_*\colon \runicov{B}{A}\rightarrow   \runicov{B}{A}\odot \rfuncatf{B}{A}.\]
\end{defn}

The relative geometric Reidemeister trace is the stable homotopy class of a map 
\[S^0\rightarrow \sh{\rfuncatf{B}{A}}\] and so it is an element of the 
$0^{th}$ stable homotopy group of $\sh{\rfuncatf{B}{A}}$.  
By definition, the relative geometric
Reidemeister trace is an invariant of the relative homotopy class of the map.

Let $X$ be a dualizable space.  For a space $U$ and a map 
\[\triangle\colon X\rightarrow 
 X\wedge U\] the {\em transfer} of an endomorphism $f\colon X\rightarrow X$
 with respect to $\triangle$ is the composite
\[S^n\stackrel{\eta}{\rightarrow} X\wedge DX\stackrel{\gamma}{\rightarrow} DX\wedge X
\stackrel{\id\wedge f}{\rightarrow}DX\wedge X\stackrel{\id\wedge \triangle}{\rightarrow}
DX\wedge X\wedge U\stackrel{\epsilon\wedge \id}{\rightarrow}S^n\wedge U.\]

Let  \[\Lambda^{f|_A}A\coloneqq\{\gamma\in A^I |f(\gamma(0))=\gamma(1)\}\] and
\[\Lambda^{f}B\coloneqq\{\gamma\in B^I |f(\gamma(0))=\gamma(1)\}.\] 
Since $A$ and $B$ are locally contractible and $f$ is taut there are neighborhoods 
$U_A$ of the fixed points of $A$ and $U_B$ of the fixed points of $B\setminus A$ 
and maps 
\[\iota_A\colon U_A\rightarrow \Lambda^{f|_A}A\]
\[\iota_B\colon U_B\rightarrow \Lambda^fB\] that take fixed points of $f$ to the constant path at that point.
Note that two fixed points of $f$ are in the same fixed point class if and only if their 
images are in the same connected component of $\Lambda^{f|_A}A$ or
 $\Lambda^fB$. See \cite{Brownbook} or
 \cite{Jiang} for the  definition of fixed point classes.

Let $\tau_{U_A}(f|_A)$ denote the transfer of $f$ with respect to the diagonal map
\[A_+\rightarrow A_+\wedge \overline{U_{A}}/\partial(\overline{U_A})\] and similarly for $B$. 

\begin{lem} \label{splitshadgeo}
If $A$ is a proper subset of $B$ there is an isomorphism 
\[\pi_0^s(\sh{\rfuncatf{B}{A}})\cong \pi_0^s(\Lambda^{f|_A}A)\oplus 
\pi_0^s(\Lambda^{f}B).\] 
The image of the relative geometric Reidemeister 
trace of $f$ under this isomorphism is 
\[(\iota_A)_*(\tau_{U_A}(f|_A))+(\iota_B)_*(\tau_{U_B}(f)).\]
\end{lem}

\begin{proof} We first define the isomorphism. 

Note that $\pi_0^s(X)\cong \bZ\pi_0(X)$ for any space $X$, so 
it is enough to show $\pi_0(\Lambda^{f|_A}A)\oplus 
\pi_0^s(\Lambda^{f}B)$ satisfies the universal property that defines 
the shadow of $\rfuncatf{B}{A}$.  

The shadow of $\rfuncatf{B}{A}$ is defined to be the coequalizer of the maps 
\[\xymatrix{ \amalg_{x,y}  \rfuncat{B}{A}(x,y)\times \rfuncat{B}{A}(f(y),x)
\ar@<.5ex>[r]\ar@<-.5ex>[r]
&\amalg_x \rfuncat{B}{A}(f(x),x)
}.\]  The inclusion maps
\[(\amalg_{x\in A}\rfuncat{B}{A}(f(x),x))
\amalg (\amalg_{x\not\in A} \rfuncat{B}{A}(f(x),x))\rightarrow 
\pi_0(\Lambda^{f|_A}A)\oplus \pi_0^s(\Lambda^{f}B)\] define a map 
\[\theta\colon \amalg_x \rfuncat{B}{A}(f(x),x)\rightarrow 
\pi_0(\Lambda^{f|_A}A)\oplus \pi_0^s(\Lambda^{f}B).\]
Let $\alpha\in  \rfuncat{B}{A}(x,y)$ and $\beta\in \rfuncat{B}{A}(f(y),x)$.
If $x,y\in A$ then $\beta\alpha$ and $f(\alpha)\beta$ represent the same elements in 
$\pi_0(\Lambda^{f|_A}A)$.  If $x,y\in B\setminus A$, $\beta\alpha$ and $f(\alpha)\beta$ 
represent the same elements in $\pi_0(\Lambda^{f}B)$.  If $x$ and $y$ do not 
satisfy these conditions, there is no condition to check on the paths.  So $\theta$ 
coequalizes.

If $\phi\colon \amalg_x \rfuncat{B}{A}(f(x),x)\rightarrow M$ is a map that coequalizes the maps
above define a map \[\bar{\phi}\colon \pi_0(\Lambda^{f|_A}A)\oplus \pi_0^s(\Lambda^{f}B)
\rightarrow  M\] by $\bar{\phi}(\gamma)=\phi(\beta)$ where $\beta$ is any element of 
$\rfuncat{B}{A}(f(x),x)$ that maps to $\gamma$ in $ \pi_0(\Lambda^{f|_A}A)\oplus \pi_0^s(\Lambda^{f}B)$.
This is independent of choices since if $\alpha$ is another lift of $\gamma$ then
there are paths $\mu$ and $\nu$ such that $f(\mu)\nu$ is homotopic to $\beta$ and 
$\nu \mu$ is homotopic to $\alpha$.  Then $\bar{\phi}$ is unique and $\pi_0(\Lambda^{f|_A}A)\oplus 
\pi_0^s(\Lambda^{f}B)$ is the coequalizer.

To describe the image of the geometric Reidemeister trace under 
this isomorphism it 
is enough to show the trace of 
\[\tilde{f|_A}\colon \tilde{A}\rightarrow \tilde{A}\] is $(\iota_A)_*(\tau_{U_A}(f|_A))$
and the trace of 
\[\tilde{f}\colon \tilde{B}/\bar{A}\rightarrow \tilde{B}/\bar{A}\] 
is $(\iota_B)_*(\tau_{U_B}(f))$.  We will describe the first, the second is similar. 

In \autoref{explicitdescription} we gave explicit descriptions of the coevaluation 
and evaluation for the dual pair $(\tilde{A}_{+}, D(\tilde{A}_+))$.  
Let $q\colon D(\tilde{A}_+)\wedge \tilde{A}_+\rightarrow DA\wedge A_+$ be the 
quotient map. If 
$\eta_1$ and $\epsilon_1$ are the coevaluation and evaluation for the dual 
pair $(\tilde{A}_{+}, D(\tilde{A}_+))$ 
and $\eta_2$ and $\epsilon_2$ are the coevaluation and evaluation for the 
dual pair $(A_+,DA_+)$ then the explicit descriptions of $\eta_1$ and $\epsilon_1$ 
show 
\[\xymatrix{S^n\ar[r]^-{\eta_1}\ar[dr]_-{\eta_2}&\tilde{A}_+\wedge_{\pi_1A}D(\tilde{A}_+)\ar[d]^q\\
&A_+\wedge DA}\] 
and 
\[\xymatrix{\sh{D(\tilde{A}_+)\wedge \tilde{A}_+\wedge \funcatf{A}}\ar[rr]^-{\epsilon_1}
\ar[d]^q&&S^n\wedge \sh{\funcatf{A}}_+\\
DA\wedge A_+\ar[r]^-{\id\wedge\triangle}&DA\wedge A_+\wedge \overline{U_A}/\partial \overline{U_A}
\ar[r]^-{\epsilon_2\wedge \id}&S^n\wedge \overline{U_A}/\partial \overline{U_A}\ar[u]^{\id\wedge \iota_A}}\]
commute.

Together these diagrams show 
\begin{eqnarray*}
(\epsilon_1\wedge \id)(\tilde{f}\wedge \id)\gamma\eta_1 &=&
(\id\wedge \iota_A)(\epsilon_2\wedge \id)(\id\wedge \triangle) q (\tilde{f}\wedge \id)\gamma\eta_1\\
&=& (\id\wedge \iota_A)(\epsilon_2\wedge \id)(\id\wedge \triangle) (f\wedge \id)\gamma q\eta_1\\
&=& (\id\wedge \iota_A)(\epsilon_2\wedge \id)(\id\wedge \triangle) (f\wedge \id)\gamma \eta_2
\end{eqnarray*}
The first composite is the trace of $\widetilde{f|_A}$.  The last composite is $(\iota_A)_*(\tau_{U_A}(f|_A))$.
\end{proof}

For a 
fixed point class $\beta$ of $f\colon B\rightarrow B$ let $i_{\beta}^{rel}$
be the index of the fixed points associated to $\beta$ that are contained in $B\setminus A$.
For a fixed point class $\alpha$ of $f|_A\colon A\rightarrow A$, let $i_\alpha$ be the 
index of the fixed points associated to $\alpha$.  Since the map $f$ is taut, 
$i_\alpha$ is the fixed point index of the fixed points in $A$ with respect to 
either $f|_A$ or $f$.  

The following corollary is a consequence of 
\autoref{splitshadgeo} and  is the generalization of \autoref{splitidx}.  This corollary 
identifies
the relative geometric Reidemeister trace with the generalization of the 
classical description of the Reidemeister trace.

\begin{cor}\label{altgeoform} Under the isomorphism in \autoref{splitshadgeo}, 
\[\rrgeo{f}{B}{A}=\left(\sum i_\alpha \alpha\right) +\left(\sum i_\beta^{rel}\beta\right)\in 
\pi_0(\Lambda^{f|_A}A)\oplus \pi_0^s(\Lambda^{f}B).\]
\end{cor} 

The following two examples were described in  \cite{NOW94}.  In that paper the 
generalized Lefschetz number and one form of the relative Nielsen number are computed.

\begin{eg}\cite[5.1]{NOW94}\label{ex51}  Let $B=D^2\times S^1$ and $A=S^1\times S^1$.
Let $f\colon B\rightarrow B$ be \[f(re^{i\theta}, e^{it})=(f_1(r)e^{-i\theta}, e^{3it})\]
where $f_1\colon [0,1]\rightarrow [0,1]$ is a continuous function such that $f_1(0)=0$, 
$f_1(1)=1$ and $f_1$ has no other fixed points.  Then $f$ is a relative map with 
six fixed points.  There are four fixed points in $A$.   These fixed points all represent 
different fixed point classes  and all have index $-1$.   The two fixed points outside of $A$  
represent different fixed point classes in $B$ and also have index $-1$.  

Since $A$ is a torus, $\pi_1(A)=\langle a,b|abab=1\rangle$.  The relation imposed on 
the shadow implies $\sh{\pi_1(A)^\phi}$ consists of 4 elements, \[1, a, b, ab.\]  
For $B$, $\pi_1(b)=\langle b \rangle$ and  $\sh{\pi_1(B)^\phi}$
consists of 2 elements, \[1,b.\]  
Then \[\rrgeo{f}{B}{A}=-1(1_A+a_A+b_A+ab_A+1_B+b_B).\]

\end{eg}  

\begin{eg}\cite[5.2]{NOW94}\label{ex52}
Let $B=S^1\times S^1$ and $A=1\times S^1$.  Let $f\colon B\rightarrow B$ be 
\[f(e^{i\theta}, e^{it})=(e^{3i\theta}, e^{4it}).\]  There are three fixed points of $f$ in $A$
and three additional fixed points of $f$ in $B\setminus A$.  

The three fixed points of $f$ in $A$ represent each of the three possible fixed point classes.
These fixed points all have index 1.  The three fixed points in $B$ that are not in $A$ also 
represent three distinct fixed point classes, but these are only three of the six possible fixed
point classes.  These fixed points also have index 1.

Let $\pi_1(B)=\langle a,b|abab=1\rangle$.  Then $\pi_1(A)=\langle a\rangle$.   The set $\sh{\pi_1(B)^\phi}$ consists
of \[1,a,a^2, b, ab, a^2b.\]  The set $\sh{\pi_1(A)^\phi}$ consists
of \[1,a,a^2.\]  Then \[\rrgeo{f}{B}{A}=1_A+a_A+a^2_A+b_B+(ab)_B+(a^2b)_B.\]

\end{eg}

\subsection*{The relative Nielsen number}
One of the expectations for the Reidemeister trace is that 
it can detect when a map has no fixed points but it does not have to provide
a lower bound for the number of fixed points.  This is very different from the Nielsen 
number.  The goal of the Nielsen number is to provide a lower bound.

In the classical case, the Nielsen number is the number of nonzero coefficients 
in the Reidemeister trace.  This implies 
the Nielsen number is zero if and only if the Reidemeister trace is zero.
For more general situations the connection between nonzero coefficients of the 
Reidemeister trace and the Nielsen number does not hold.  It remains true that 
the Nielsen number is zero if and only if the Reidemeister trace is zero.

The inclusion of $A$ into $B$ induces a map $\pi_1(A)\rightarrow \pi_1(B)$ and 
also induces a map $\Phi$ from the fixed point classes of $A$ to the fixed point
classes of $B$.  A fixed point class of $f$ or $f|_A$ is essential if its coefficient
in the classical Reidemeister trace is nonzero.  
Let \[N(f,f|_A)\] be the number of essential fixed point classes of $B$ that
are in the image of an essential class of $A$.  Let $N(f)$ be the classical 
Nielsen number of $f$ and $N(f|_A)$ be the classical Nielsen number of 
$f|_A$.

\begin{defn}\cite[2.5]{Zhao05} 
The \emph{relative Nielsen number}, $N(f;B,A)$, is \[N(f|_A)+\left( N(f)
-N(f,f|_A)\right) .\] 
\end{defn}

\begin{lem} The relative Nielsen number of $f$ is zero if and only if 
the relative geometric Reidemeister trace of $f$ is zero.
\end{lem}

\begin{proof} If the relative geometric Reidemeister trace of $f$ is 
zero \autoref{altgeoform} implies
$\left(\sum i_\beta^{rel}\beta\right)+\left(\sum i_\alpha \alpha\right)$
is zero.  Since $\mathbb{Z}\sh{\rfuncatf{B}{A}}$ is a free group generated by
the $\alpha$'s and $\beta$'s each $ i_\beta^{rel}$ and $i_\alpha$ are zero.  
Since the $i_\alpha$'s
are zero, $N(f|_A)$  and $N(f,f|_A)$ are zero.  Since each of the $i_\alpha$'s are zero
$i_\beta=i_\beta^{rel}=0$ for every $\beta$.  This implies $N(f)$ is also zero.

By definition $N(f|_A)$, $N(f)$, and $N(f,f|_A)$ are all greater than or 
equal to zero and  $N(f,f|_A)\leq N(f)$.  If the relative Nielsen number of $f$ is zero
$N(f|_A)=0$ and $N(f)=N(f,f|_A)$.  Since $N(f|_A)=0$, $N(f,f|_A)=0$
and so $N(f)=0$.  Since $N(f|_A)=0$ all of the $i_\alpha$'s are zero and $i_\beta^{rel}
=i_\beta$.  Since $N(f)=0$, $i_\beta=0$ for all $\beta$. 
\end{proof}

The relative Nielsen numbers for the maps in the examples above were
computed in \cite{NOW94}.  The relative Nielsen number for \autoref{ex51}
is 4.  This is not the number of non zero coefficients in the relative Reidemeister trace.
The relative Nielsen number for \autoref{ex52} is 6.  This does happen to be 
the number of nonzero coefficients in the relative Reidemeister trace.  These numbers
coincide because $N(f,f|_A)$ is zero for this example.

Other references for relative Nielsen theory include \cite{Jezierski95, 
schirmer86,Schirmer88, Zhao99, Zhao01}.
These invariants are also related to the Nielsen numbers for
stratified spaces defined in \cite{Jiang07}.

%% file: relative_fp_6.tex
\section{The global Reidemeister trace}\label{algsec}
In this section we define the relative global Reidemeister trace.  This invariant 
is a generalization of the relative global Lefschetz number and can be 
identified with the relative geometric Reidemeister trace.   The relative global Reidemeister
trace a relative generalization of the invariant defined in \cite{Husseini}.
It is related to the  invariants defined in 
\cite{NOW94, Zhao06}, but it is not the same as either of these invariants.

Let $\mathbb{Z}\rfuncat{B}{A}$ be the category with the same objects as $\rfuncat{B}{A}$.  The 
morphism set \[\bZ\rfuncat{B}{A}(x,y)\] is the free abelian group on 
the set $\rfuncat{B}{A}(x,y)$.

There is a functor \[C_*(\runicov{B}{A})\colon \mathbb{Z}\rfuncat{B}{A}^\op\rightarrow \Ch_\bQ\] 
defined by $C_*(\runicov{B}{A})(x)=C_*(\runicov{B}{A}(x);\bQ)$ where the second
$C_*$ indicates the cellular chain complex.  The action of the morphisms of $\rfuncat{B}{A}$ is 
induced from the action on $\runicov{B}{A}$. 
This functor is defined in the same way that the functor $H_*(
\rcomsp{B}{A})$ is defined from the functor $\rcomsp{B}{A}$ except we 
replace the rational homology functor with the rational chain complex functor.  

\begin{prop} If $A\subset B$ are closed smooth manifolds
the $\mathbb{Z}\rfuncat{B}{A}$-module $C_*(\runicov{B}{A})$ is dualizable. 
\end{prop}

\begin{proof}   Like \autoref{algspacedual} there are two possible proofs of this 
theorem.   

The rational cellular chain complex functor induces a functor
on bicategories, and for $A\subset B$ closed smooth 
manifolds, \autoref{runivcovdual} shows that $\runicov{B}{A}$ is dualizable.  
\autoref{tracefunctor}  then implies that $C_*(\runicov{B}{A})$ is dualizable.

There is a second approach using \autoref{explicitdual}.
If $x\in A$, $C_*(\runicov{B}{A})(x)=C_*(\tilde{A}_x)$ as a module over $\pi_1(A,x)$.
This is a finitely generated free module and so is dualizable with dual 
\[\Hom_{\mathbb{Z}\pi_1(A,x)}(C_*(\tilde{A}_x),\mathbb{Z}\pi_1(A,x)).\]
If $x\in B\setminus A$, $C_*(\runicov{B}{A})(x)=C_*(\tilde{B}_x/\bar{A}_x)$ as a 
module over $\pi_1(B,x)$.  This is also a finitely generated free module and so 
is dualizable with dual \[\Hom_{\mathbb{Z}\pi_1(B,x)}(C_*(\tilde{B}_x/\bar{A}_x),
\mathbb{Z}\pi_1(B,x)).\]  Since $C_*(\runicov{B}{A})$ is supported on 
isomorphisms, \autoref{explicitdual} implies $C_*(\runicov{B}{A})$ is 
dualizable.
\end{proof}

A map $f\colon (B,A)\rightarrow (B,A)$ induces a map 
\[f_*\colon C_*(\runicov{B}{A})\rightarrow C_*(\runicov{B}{A})\odot \mathbb{Z}
\rfuncatf{B}{A}.\]
Since $C_*(\runicov{B}{A})$ is dualizable, the trace of $f_*$ is defined.

\begin{defn} 
The \emph{relative global Reidemeister trace} of $f\colon (B,A)\rightarrow (B,A)$, 
$\rralg{f}{B}{A}$, is the trace of 
\[f_*\colon C_*(\runicov{B}{A})\rightarrow C_*(\runicov{B}{A})\odot \mathbb{Z}
\rfuncatf{B}{A}.\]
\end{defn}

The relative global Reidemeister trace of $f$ is a map 
\[\bZ\rightarrow \bZ\sh{\rfuncatf{B}{A}}.\]

\begin{lem} If $A$ is a proper subset of $B$ then
\[\sh{\rfuncatf{B}{A}}\cong \sh{\funcatf{B}}\amalg \sh{\funcatf{A}}.\]

 The image of 
$\rralg{f}{B}{A}$ under this isomorphism is 
\[\sum_{x\in\sh{\comcatf{A}}}\ralg{f|_{A(x)}}[x]+
\sum_{y\in\sh{\comcatf{B}}}\ralg{f|_{B(y)/(B(y)\cap A)}}[y].\] 
\end{lem}

Here $\ralg{f|_{A(x)}}$ denotes the usual global Reidemeister trace
of $f|_{A(x)}$ as defined by Husseini in \cite{Husseini}.  The invariant $\ralg{f|_{B(y)}\cup Cf|_{B(y)\cap A}}$ 
is the  trace of
\[f_*\colon C_*({B(y)}/(B(y)\cap A))\rightarrow C_*({B(y)}/(B(y)\cap A))
\otimes \pi_1^f(B,y)\]
as a module over $\pi_1(B,y)$.

\begin{proof}To simplify notation, consider the case where $A$ and $B$ are connected.
The proof is similar if $A$ and $B$ are not connected.

The shadow is defined to be the coequalizer of the maps 
\[\xymatrix{ \amalg_{x,y}  \rfuncat{B}{A}(x,y)\times \rfuncat{B}{A}(f(y),x)
\ar@<.5ex>[r]\ar@<-.5ex>[r]
&\amalg_x \rfuncat{B}{A}(f(x),x)
}.\]
Instead of indexing these coproducts over all objects in $\rfuncat{B}{A}$ we can index 
over representatives of each isomorphism
class of objects in $\rfuncat{B}{A}$.  This gives four terms in the first coproduct.  The two 
cross terms are both empty and so this coequalizer splits into the coequalizer that 
defines $\sh{\funcatf{B}}$ and the coequalizer that defines $\sh{\funcatf{A}}$.

For the second statement, note that this isomorphism is compatible with 
the description of the dual pair.  Then the trace is the pair of classical traces.
\end{proof}

This description of the relative global Reidemeister trace shows that the second 
coordinate is the the relative Reidemeister trace of \cite{Zhao06}.  This also shows
that this invariant is related to, but not the same as,  the generalized Lefschetz number defined
 in \cite{NOW94}. 

\begin{prop}\label{alggeocompare}
If $A\subset B$ are closed smooth manifolds and 
$f\colon (B,A)\rightarrow (B,A)$ is a relative map then 
\[\rrgeo{f}{B}{A}=\rralg{f}{B}{A}.\]
\end{prop}

\begin{proof}
Since both $\rrgeo{f}{B}{A}$ and $\rralg{f}{B}{A}$ are defined as 
traces and the rational cellular chain complex functor is strong
symmetric monoidal this proposition follows from \autoref{tracefunctor}.
\end{proof}

%% file: relative_fp_8.tex
\section{A converse to the relative Lefschetz fixed point theorem}\label{converse}

There are several proofs of the converse to the relative Lefschetz fixed 
point theorem.  Some,  like
\cite{Jezierski95, schirmer86, Schirmer88, Zhao01}, are generalizations
of the simplicial arguments used in the standard proof of the converse
to the classical Lefschetz fixed point theorem, see \cite{Brownbook}.

In this section and the next section we describe a proof of the converse to the relative Lefschetz fixed 
point theorem that follows the outline of \cite{KW, KW2}.  This approach is not simplicial
and it easily generalizes.  For example, see \cite{KW2} for the equivariant generalization
and \cite{thesis} for the fiberwise generalization.  

The approach of \cite{KW} is based on invariants that detect sections of fibrations. 
In the next section we prove relative generalizations of the results in \cite{KW}.  In this 
section we apply those results to relative fixed point invariants. 

The main result of this section is:

\begin{thmB}[The converse to the Relative Lefschetz fixed point theorem] 
Suppose $A\subset B$ are closed smooth manifolds of dimension at least 3 and
the codimension of $A$ in $B$ is at least 2.  
The relative global Reidemeister trace of a map \[f\colon (B,A)\rightarrow 
(B,A)\] is zero if and only if $f$ is relatively homotopic to 
a map with no fixed points. 
\end{thmB}

The first step in the proof of Theorem B is to describe relative maps without fixed points 
in terms of relative sections.

\begin{lem}\label{maptolift} Let $A\subset B$ be closed smooth manifolds.  
Relative homotopies of a map $f\colon (B,A)\rightarrow (B,A)$  to a relative map 
with no fixed points
correspond to liftings that make 
the diagram below commute up to relative homotopy 
\[\xymatrix{&(B\times B \setminus  \triangle, A\times A \setminus \triangle)\ar[d]\\
(B,A)\ar[r]^{\Gamma_f}\ar@{.>}[ur]&(B\times B,A\times A).}\]
\end{lem}

The function $\Gamma_f$ is the graph of $f$ and $
\Gamma_f(m)=(m,f(m))$.

\begin{proof}
If $f$ is relatively homotopic to a fixed point free map $g$  via a relative 
homotopy $H$, then $\Gamma_H$ is a relative homotopy from $\Gamma_f$ to  
$\Gamma _g$.

For the converse, suppose there is a relative map \[k\colon (B,A)\rightarrow 
(B\times B\setminus \triangle, A\times A\setminus \triangle)\] and a relative homotopy 
$K$ from $k$ to $\Gamma_f$.  

If $A$ is a smooth manifold the first coordinate projection 
\[\proj_1\colon A\times A\setminus \triangle\rightarrow A\] is a fiber bundle and there is a lift $J_A$ in 
the diagram 
\[\xymatrix{A\ar[r]^-k\ar[d]^{i_0}&A\times A\setminus \triangle\ar[d]^{\proj_1}\\
A\times I\ar[r]^-{\proj_1K}\ar@{.>}[ur]^{J_A}&A.}\] 

Since $A\subset B$ is a cofibration and $\proj_1\colon B\times B\setminus \triangle \rightarrow 
B$ is a fibration the diagram 
\[\xymatrix@C=3cm{B \cup A\times I \ar[r]^-{k\cup J_A}\ar[d]^{i_0}
&B\times B\setminus \triangle\ar[d]^{\proj_1}\\
B\times I \ar[r]_-{\proj_1\circ K}\ar@{.>}[ur]^-J&B}\]
has a lift $J$ extending the lift $J_A$ above, see \cite[Theorem 4]{Strom1966}.
Note that  $\proj_1\circ J(-, 1)=\id$.  Let $g=\proj_2J(-, 1)$.
This map has no fixed points.

The homotopies $K$ and $J$ define a relative homotopy from 
$\Gamma_f$ to $\Gamma_g$.
\end{proof}

Given a map $f\colon V\rightarrow Y$, let $\fibm(f)\colon \fibs(f)\rightarrow Y$ denote a Hurewicz fibration such that 
\[\xymatrix{V\ar[r]\ar[dr]_f&\fibs(f)\ar[d]^-{\fibm(f)}\\
&Y}\] commutes and $V\rightarrow \fibs(f)$ is an equivalence.

\begin{lem}\label{lifttosect}
Let $X\subset Y$,  $p\colon M_Y\rightarrow
Y$ be a space over $Y$ and $M_X\subset p^{-1}(X)$.  

Liftings up to relative homotopy in the diagram  \[\xymatrix{&(M_Y,M_X)\ar[d]^f\\
(B,A)\ar@{.>}[ur]\ar[r]^g&(Y,X)}\] correspond to relative sections of the pair of fibrations 
\[(g^*\fibs(f_Y),g^*\fibs(f_X)) \rightarrow (B,A).\] 
\end{lem}

If $p\colon E\rightarrow B$ is a Hurewicz fibration the \emph{unreduced fiberwise 
suspension} of $p$ is the double mapping cylinder \[S_BE\coloneqq B\times \{0\}\cup_p E\times [0,1]
\cup_p B\times \{1\}.\]  The map $p\colon E\rightarrow B$ defines a fibration 
\[S_BE\rightarrow B.\]  There are two sections of this fibration, $\sect_1, 
\sect_2\colon B\rightarrow S_BE$, defined by the inclusions of $B\times \{0\}$ and $B\times \{1\}$.
If $S_B^0\coloneqq B\amalg B$, these sections define an element of 
\[[S_B^0, S_BE]_B.\]

Let $i_B\colon B\times B\setminus \triangle\rightarrow B\times B$ be the inclusion.
Then the pair of fibrations \[(\Gamma_{f*}(N(i_B)),\Gamma_{f*}(N(i_A))) \rightarrow (B,A)\] 
determine an element in 
\[[S_B^0, S_B\Gamma_{f*}(N(i_B))]_{B}\oplus
[S_A^0, S_A\Gamma_{f*}(N(i_A))]_{A}.\]  This element will be denoted $\rrkw{f}{B}{A}$.

\begin{prop}\label{relconverse}
Let  $A\subset B$ be closed smooth manifolds of dimension at least $3$
such that the codimension of $A$ in $B$ is at least 2. 
A continuous map \[f\colon (B,A)\rightarrow (B,A)\]
 is relatively homotopic to a map with no fixed points 
if and only if $\rrkw{f}{B}{A}=0$.
\end{prop}

The proof of this proposition, except for one key step proved in the next section, 
follows the preliminary lemma below.

\begin{lem}\cite[6.1, 6.2]{KW}\label{connectivity}
Let $M$ be a manifold of dimension $n$ and $i\colon M\times M\setminus 
\triangle\rightarrow M \times M$ be the inclusion.  Then $\Gamma_{f*}(N(i))\rightarrow M$ 
is $(n-1)$-connected.
\end{lem}

\begin{proof}[Proof of \autoref{relconverse}]
\autoref{maptolift} and \autoref{lifttosect} convert the question
of finding a lift of a relative map $f\colon (B,A)\rightarrow (B,A)$ to the 
question of finding a section of the fibration
\[(\Gamma_{f*}(N(i_B)),(\Gamma_{f|_A*}(N(i_A)))
\rightarrow (B,A).\]

If the dimension of $A$ is $n_A$ and the dimension of $B$ is $n_B$ then 
\autoref{connectivity} implies that $\Gamma_{f*}(N(i_B))
\rightarrow B$ is $(n_B-1)$-connected and $\Gamma_{f|_A*}(N(i_A))
\rightarrow A$ is $(n_A-1)$-connected.  If $n_A$ and $n_B$ are at least 
3 and $n_B-n_A$ is at least 2, 
\autoref{chitochi} implies that \[(\Gamma_{f*}(N(i_B)),\Gamma_{f|_A*}(N(i_A)))
\rightarrow (B,A)\] has a relative section 
if and only if $\rrkw{f}{B}{A}=0$.
\end{proof}

The hypotheses in this proposition are not the standard hypotheses used
in the converse to the relative Lefschetz fixed point theorem.  The standard
condition is that \[\pi_1(B\setminus A)\rightarrow \pi_1(B)\]
is surjective.  The codimension condition implies this condition.
We use a codimension condition since it is  compatible with 
the approach of the proof. I don't know if the surjectivity condition can be used in this 
approach.  

To complete the proof of Theorem B we need to compare $\rrkw{f}{B}{A}$ and 
the relative geometric Reidemeister trace.

\begin{prop}\label{htpykwcompare} Let $A\subset B$ be closed smooth 
manifolds and $f\colon (B,A)\rightarrow (B,A)$ be
a relative map.  Then  
\[\rrkw{f}{B}{A}=0 \,\,\mathrm{if \, and \,only\, if}\,\,\rrhtpy{f}{B}{A}=0.\]
\end{prop}

We recall a lemma from \cite{KW2}.

\begin{lem} \cite[7.1]{KW2}\cite[8.3.1]{thesis}\label{fibidentify}
Let $M$ be a closed smooth manifold with normal bundle 
$\nu_M$.   Then there is a weak equivalence
\[S^{\nu_M}\odot \Gamma_{f*}S_{M\times M}\fibs (i_M)
\rightarrow S^n\wedge \Lambda^fM.\] 
\end{lem}

\begin{proof}[Proof of \autoref{htpykwcompare}]  
Suppose $X$ and $Y$ are ex-spaces over $B$, the projection maps are fibrations and 
the sections are fiberwise cofibration.  Let $\{X, Y\}_B$ denote
the fiberwise stable homotopy classes of maps from $X$ to $Y$.

If $A$ and $B$ are both closed smooth manifolds of dimension at least 
three, then 
the dimension assumption, \autoref{connectivity}, and the fiberwise 
Freudenthal suspension theorem in \cite[4.2]{james96} imply
that the maps \[[S^0_A,S_A\Gamma_{f|_A*}(N(i_A))]_A\rightarrow \{S^0_A,S_A\Gamma_{f*}(N(i_A))\}_A\]
\[[S^0_B,S_B\Gamma_{f*}(N(i_B))]_B\rightarrow \{S^0_B,S_B\Gamma_{f*}(N(i_B))\}_B\]
are isomorphisms.   
Costenoble-Waner duality 
\cite[18.5.5, 18.6.3]{MS} and \autoref{fibidentify} imply there are isomorphisms
\begin{eqnarray*}
 \{S^0_A,S_A\Gamma_{f|_A*}(N(i_A))\}_A&\cong& \{S^n, S^{\nu_A}\odot S_A\Gamma_{f*}(N(i_A))\}\\
&\cong& \{S^n,S^n\wedge \Lambda^{f|_A}A_+\}.
\end{eqnarray*}  
and 
\begin{eqnarray*}
 \{S^0_B,S_B\Gamma_{f*}(N(i_B))\}_B&\cong& \{S^n, S^{\nu_B}\odot S_B\Gamma_{f*}(N(i_B))\}\\
&\cong& \{S^n,S^n\wedge \Lambda^fB_+\}.
\end{eqnarray*}  

Let $U_A$ be a neighborhood of the fixed points of $f|_A$ such 
that there is a map \[\iota_A\colon U_A\rightarrow \Lambda^{f|_A}A\] that takes fixed 
points to the constant path at that point.
In \cite[6.3.2]{thesis} it is shown that the image of $\rrkw{f}{B}{A}$  in $\pi_0^s( \Lambda^{f|_A}A_+)$
is $\iota_A(\tau(f|_{U_A}))$.  

Let $U_B$ be a neighborhood of the fixed points of $f$ in $B\setminus A$ 
such that there is a map
\[\iota_B\colon U_B\rightarrow \Lambda^fB\] which 
takes fixed points to constant paths. 

The image of $\rrkw{f}{B}{A}$ in
$\pi_0^s( \Lambda^{f}B_+)$ is the composite of the transfer of $f$ with 
respect to the diagonal map \[B_+\rightarrow B_+\wedge \overline{(U_B\amalg U_A)}/\partial
\overline{(U_B\amalg U_A)}\] with the map 
\[\iota\coloneqq \iota_A\amalg\iota_B\colon U_A\amalg U_B\rightarrow \Lambda^fB.\]  Since the transfer 
is additive, \cite{dold76}, the image of $\rrkw{f}{B}{A}$ in $\pi_0^s( \Lambda^{f}B_+)$ is 
\[\iota(\tau_{U_B\amalg U_A}(f))=\iota(\tau_{U_B}(f)+\tau_{U_A}(f|_A))
=\iota(\tau_{U_B}(f))+\iota(\tau_{U_A}(f|_A)).\]

Then $\rrkw{f}{B}{A}$ is zero if and only if $\iota_A(\tau(f|_{U_A}))$ and 
$\iota(\tau_{U_B}(f))+\iota(\tau_{U_A}(f|_A))$ are both zero.  Using 
\autoref{splitshadgeo} these elements are zero if and only if $\rrhtpy{f}{B}{A}$
is zero. 
\end{proof}

\begin{proof}[Proof of Theorem B]
\autoref{relconverse} implies that $f$ is relatively homotopic to a 
fixed point free map if and only if $\rrkw{f}{B}{A}=0$.  
\autoref{htpykwcompare} implies $\rrkw{f}{B}{A}=0$ if and only if 
$\rrgeo{f}{B}{A}=0$.
\autoref{alggeocompare} implies $\rralg{f}{B}{A}=\rrgeo{f}{B}{A}$.
\end{proof}

\begin{rmk} \autoref{relconverse} and the proof of \autoref{htpykwcompare} show if
$\dim(A)\geq 3$ and $\dim(B)\geq \dim(A)+2$  $\rrkw{f}{B}{A}$ is zero if and only
if the two {\bf nonrelative} invariants for $A$ and $B$ are zero.  

Using these two invariants to define a relative invariant would be analogous to 
defining the relative invariants in the previous sections as the pair of classical 
invariants for the spaces $A$ and $B$.  This alternate definition would satisfy 
the requirements of the introduction for a fixed point invariant.  However, there are several reasons why  
the corresponding definition in the equivariant case is not acceptable.  The definitions
in the previous sections were chosen because they are consistent with the choices 
in \cite{equiv}.
\end{rmk}

\section{Relative sections}\label{relsec}
In this section we generalize the result from \cite{KW} on sections of fibrations 
to relative fibrations.

If the dimension of $B$ is $2n$ and the fibration $p\colon E\rightarrow B$ is 
$n+1$-connected, it is shown in \cite{KW} 
that the two sections 
\[\sect_1, \sect_2\colon B\rightarrow S_BE\]
are homotopic over $B$ if and only there is a section of $p$.  
We can generalize this result to relative sections.  

If $A\subset B$ let $E_A$ be a subspace of $E_B$ such that the image of $p$ restricted to $E_A$ is 
contained in $A$.
Let $S_{A,B}E_A$ be \[B\times \{0\}\cup E_A\times I\cup A\times \{1\}.\] Let 
$ [(S^0_B,A\amalg B), (S_BE_B,S_{A,B}E_A)]_B$ be the relative fiberwise homotopy 
classes of maps from $(S^0_B,A\amalg B)$ to $(S_BE_B,S_{A,B}E_A)$.

\begin{defn} Let $A\subset B$, $p\colon E_B\rightarrow B$ be a fibration, and 
$E_A\subset p^{-1}(A)$ such that $E_A\rightarrow A$ is a fibration.
The \emph{relative homotopy Euler class} 
\[\chi\in [(S^0_B,A\amalg B), (S_BE_B,S_{A,B}E_A)]_B\] is 
$\sect_1\amalg \sect_2\colon S_B^0\rightarrow S_BE_B$.
\end{defn}

\begin{prop}\label{secttosect} 
If $(E_B,E_A)\rightarrow (B,A)$ admits a relative section $\sectt$ then $\chi$ is 
trivial.

Conversely, assume $p\colon E_A\rightarrow A$ is $(m+1)$-connected,
$A$ is a $2m$-dimensional CW-complex, $p\colon E_B\rightarrow B$ is 
$(n+1)$-connected and $(B,A)$ is a relative $2n$-dimensional CW-complex.
If $\chi$ is trivial then $p$ has a relative section.
\end{prop}

Before we prove this proposition we recall a preliminary lemma.

\begin{lem}\label{htpypullback}\cite[3.1]{KW} Let $p\colon E\rightarrow B$ be a $(j+1)$-connected
fibration and $P$ be the homotopy pullback \[\xymatrix{P\ar[r]\ar[d]&B\ar[d]\\
B\ar[r]&S_BE.}\]  A fiberwise homotopy from 
$\sect_1$ to $\sect_2$ defines a $2j$-equivalence 
$q\colon E\rightarrow P$.
\end{lem}

\begin{proof}[Proof of \autoref{secttosect}]
If there is a relative section $\sectt$ then the homotopy \[H\colon (S_B^0,A\amalg B)\times I
\rightarrow (S_BE_B, S_{A,B}E_A)\] defined by 
$H(b,t)=(\sectt(b),t)$ shows $\chi$ is trivial.

If $\chi$ is trivial there is a relative fiberwise homotopy \[K\colon (S_B^0,A\amalg B) \times I 
\rightarrow (S_BE_B, S_{A,B}E_A)\] from $\sect_2$ to $\sect_1$.   
The restriction of 
$K$ to $S^0_A$ defines a homotopy between $\sect_1|_A\colon A\rightarrow 
S_AE_A$ and $\sect_2|_A$. 
\autoref{htpypullback}, Whitehead's theorem, and the homotopy $K|_{S^0_A}$ imply 
\[q_{A*}\colon [A,E_A]\rightarrow [A, P_A]\] is a bijection.   The space $P_A$ is as in 
\autoref{htpypullback}.

The restriction $K|_{S^0_A}$ induces a map 
$h_A\colon A\rightarrow P_A$ such that $p h_A=\id$.  Since $q_{A*}$ is a bijection 
there is a map $k_A\colon A\rightarrow E_A$  and a homotopy $J_A$ from 
$q_A k_A$ to $h_A$.  Then $p k_A=p(q_Ak_A)\simeq ph_A=\id_A$ via 
the homotopy $p(J_A)$.  The diagram 
\[\xymatrix{A\ar[r]^-{k_A}\ar[d]^{i_0}&E_A\ar[d]^p\\
A\times I \ar[r]^-{p(J_A)}\ar@{.>}[ur]^{L_A}&A}\] has a lift $L_A$, and 
$p(L_A(a,1))=a$.  Then $L_A(-,1)$  is a section of $p^{-1}(A) \rightarrow A$ that
is contained in $E_A$.

The homotopy $K$ defines a map $h_B\colon B
\rightarrow P_B$ extending the map $h_A$.    The space $P_B$ is as in 
\autoref{htpypullback}.
The homotopy extension and lifting property implies the dotted maps 
in the following diagram can be filled in
\[\xymatrix{A\ar[rr]^{i_1}\ar[dd]&&A\times I \ar[dd]\ar[dl]^{J_A}&&A\ar[ld]^{k_A}\ar[dd]\ar[ll]_{i_0}\\
&P_B&&E_B\ar'[l]_{q}[ll]\\
B\ar[rr]^{i_1}\ar[ur]^{h_B}&&B\times I \ar@{.>}[ul]^{J_B}&&B\ar[ll]_{i_0}\ar@{.>}[ul]^{k_B}}\]
defining maps $k_B$ and $J_B$ extending $k_A$ and $J_A$.

Since the pair $(B,A)$ has the relative homotopy lifting property there 
is a lift $L_B$ in the diagram
\[\xymatrix@C=1.5cm{B\cup (A\times I)\ar[r]^-{k_B\cup L_A}\ar[d]&E_B\ar[d]^p\\
B\times I \ar[r]^-{p(J_B)}\ar@{.>}[ur]^-{L_B}&B}.\]
Evaluating at $1$, $p(L_B(b,1))=pJ_B(b,1)=ph_B(b)=b$.  Since 
 $L_B(a,1)\in E_A$ for $a\in A$, $L_B(-,1)$ is the 
required section.
\end{proof}

\autoref{maptolift}, \autoref{lifttosect}, and \autoref{secttosect} imply 
$\chi$ is a complete obstruction to determining if a relative fibration has 
a section.  In the examples we are interested in, it is easier to work with invariants
defined from $\chi$ than with $\chi$ itself.  Under some additional hypotheses, 
these associated invariants are zero if and only if $\chi$ is zero. 

If $A\subset B$, define \[C_B(B,A)\coloneqq B\times\{0\} \cup A\times [0,1] \cup B\times \{1\}.\] 
This is an ex-space over $B$ with section given by the inclusion of $B$
into $C_B(B,A)$ as $B\times \{0\}$.

In the diagram below the vertical maps are induced by cofiber sequences, \cite[II.2.4]{CJ} and 
so the columns are exact.  The horizontal maps are forgetful maps.  The diagram commutes.
\[\xymatrix@C=3pt{\chi_{B,A}\ar@{}@<-1ex>[r]^-\in
&[(C_B(B,A), S_BA),(S_BE_B,S_AE_A)]_B\ar[d]^-\phi\ar[rr]^-\psi
&&[C_B(B,A), S_BE_B]_B\ar[d]^-\rho
&\bar{\chi}_{B,A}\ar@{}@<-1ex>[l]^-\ni\\
\chi\ar@{}@<-1ex>[r]^-\in
&[(S_B^0,S_A^0), (S_BE_B,S_AE_A)]_B\ar[rr]\ar[d]&&
[S_B^0,S_BE_B]_B\ar[d]
&\chi_B\ar@{}@<-1ex>[l]^-\ni\\
\chi_A\ar@{}@<-1ex>[r]^-\in
&[A\amalg B,S_{A,B}E_A]_B\ar[rr]&&
[A\amalg B,S_BE_B]_B
&\bar{\chi}_A\ar@{}@<-1ex>[l]^-\ni}\]
The elements $\chi_A$, $\chi_B$, and $\bar{\chi}_A$ are the images of $\chi$.  The element $\chi_{B,A}$
is defined if $\chi_A=0$.  Then $\chi_{B,A}$ is the preimage of $\chi$.  The element 
$\bar{\chi}_{B,A}$ is defined if $\bar{\chi}_A=0$.  Then $\bar{\chi}_{B,A}$ is the preimage of $\chi_B$.

\begin{lem}\label{injectchi}  If $\bar{\chi}_{B,A}=0$ then $\chi_{B,A}=0$. 
\end{lem}

\begin{proof} Suppose $\bar{\chi}_{B,A}=0$.  Then there is a fiberwise homotopy 
\[L\colon C_B(B,A)\times I
\rightarrow S_BE_B\] such that
\begin{eqnarray*}
L(b,1,0)&=&\sect_2(b)\\
L(b,1,1)&=&\sect_1(b)\\
L(b,0,t)&=&\sect_1(b)\\
L(a,s,0)&=&\chi_{B,A}(a,s) \in S_AE_A\\
L(a,s,1)&=&\sect_1(a)
\end{eqnarray*}
for all $a\in A$, $b\in B$, and $s,t\in I$.

Let $J\coloneqq (\{0\}\times I) \cup (I\times \{1\}) \cup (\{1\}\times I)$.  Define
a map \[\bar{L}\colon B\times J\rightarrow S_BE_B\] by 
\begin{eqnarray*}
\bar{L}(b,0,t)&=&\sect_1(b)\\
\bar{L}(b,s,1)&=&\sect_1(b)\\
\bar{L}(b,1,t)&=&L(b,1,t).
\end{eqnarray*}
The diagram 
\[\xymatrix@C=40pt{(B\times J)\cup_i (A\times I\times I)\ar[r]^-{\bar{L}\cup L|_{A\times I\times I}}
\ar[d]&S_BE_B\ar[d]\\
B\times I \times I \ar@{.>}[ur]^K\ar[r]_-{\mathrm{proj}}&B}\] commutes and there is a lift $K$ since
$S_BE_B\rightarrow B$ is a fibration.   

Then $K_0\coloneqq K(-,-,0)\colon B\times I \rightarrow S_BE_B$ satisfies 
\begin{eqnarray*}
K_0(b,0)&=&K(b,0,0)=L(b,0,0)=\sect_1(b)\\
K_0(b,1)&=&K(b,1,0)=L(b,1,0)=\sect_2(b)\\
K_0(a,s)&=&K(a,s,0)=L(a,s,0)\in S_AE_A
\end{eqnarray*}
Define a map \[\tilde{K}\colon C_B(B,A)\times I \rightarrow S_BE_B\]
by 
\begin{eqnarray*}
\tilde{K}(b,1,t)&=&K_0(b, 1-t) \\
\tilde{K}(b,0,t)&=& \sect_1(b)\\
\tilde{K}(a,s,t)&=& K_0(a, s(1-t)) 
\end{eqnarray*}
$\tilde{K}$ shows $\chi_{B,A}$ is trivial in $[(C_B(B,A), S_BA),(S_BE_B,S_AE_A)]_B$.
\end{proof}

\begin{lem}\label{bypass} If the map $E_B\rightarrow B$ is a $(\mathrm{dim}(A) +1)$-equivalence then 
$\rho$ is injective.
\end{lem}

\begin{proof} In this proof let $i$ denote the inclusion of $A$ in $B$.

 Let $\Sigma_B(A\amalg B)\coloneqq ((A\times I)\amalg B)/\sim$ 
where $(a,0)\sim i(a)\sim (a,1)$.  
Then $\rho$ is part of a long exact sequence 
\[\xymatrix@C=10pt{[\Sigma_B (A\amalg B), S_BE_B]_B\ar[r]&
 [C_B(B,A), S_BE_B]_B\ar[r]^-\rho&
[S_B^0,S_BE_B]_B\ar[r]&
[A\amalg B,S_BE_B]_B.
}\] 
To show that $\rho$ is injective it is enough to show 
\[[\Sigma_B (A\amalg B), S_BE_B]_B\] is trivial.

Let $\alpha$ be an element of $[\Sigma_B (A\amalg B), S_BE_B]_B$.
Then $\alpha$ defines a map $S^1\times A\rightarrow S_BE_B$ also 
denoted $\alpha$.  This map satisfies $p\alpha(t,a)=i(a)$.  Consider
the diagram 
\[\xymatrix{S^1\times A\ar[rr]^{i_0}\ar[dd]&&
S^1\times A\times I\ar'[d][dd]\ar[dl]_{i\circ\mathrm{proj}}&&
S^1\times A\ar[dd]\ar[ll]_{i_1}\ar[dl]^\alpha\\
&B&&S_BE_B\ar[ll]\\
D^2\times A\ar[ur]^{i\circ \mathrm{proj}}\ar[rr]^{i_0}&&D^2\times A\times I\ar@{.>}[ul]^H&&
D^2\times A\ar[ll]_{i_1}\ar@{.>}[ul]^\beta
}\]
Since $S_BE_B\rightarrow B$ is a $(\mathrm{dim}(A)+2)$-equivalence, the homotopy extension 
and lifting property implies there are maps $\beta$ and $H$ 
that make the diagram commute.

The diagram 
\[\xymatrix@C=50pt{(D^2\times A)\amalg_i S^1\times A\times I
\ar[r]^-{\beta \amalg (\alpha\circ \mathrm{proj})}\ar[d]&S_BE_B\ar[d]^p\\
D^2\times A\times I \ar[r]^H\ar@{.>}[ur]^K&B
}\] commutes.  
Since $S_BE_B\rightarrow B$ is a fibration there is a lift $K$ that makes the diagram commute.
Then \[K_0\coloneqq K(-, -, 0)\colon D^2\times A\rightarrow S_BE_B\] satisfies
\[pK_0(v,a)=H(v,a,0)=i(a)\]
and \[K_0(w,a)=\alpha(w,a)\] if $w\in S^1$.  Then
\[K_0\amalg \id \colon ((D^2\times A)\amalg B)/\sim\rightarrow S_BE_B\]
defines a map that shows $\alpha$ is trivial.
\end{proof}

The following proposition is a consequence of 
\autoref{secttosect}, \autoref{injectchi}, and \autoref{bypass}.

\begin{prop}\label{chitochi} If
$p\colon E_A\rightarrow A$ is $(m+1)$-connected,
$A$ is a $2m$-dimensional CW-complex, $p\colon E_B\rightarrow B$ is 
$(2m+1)$-connected and $(B,A)$ is a relative $4m$-dimensional CW-complex
$(E_B,E_A)\rightarrow (B,A)$ admits a relative section if and only if 
$\chi_A$ and $\chi_B$ are both zero.
\end{prop}

%% file: relative_fp_9.tex
\section{Other descriptions of $\odot$ in special cases}\label{extracat}

These are the proofs omitted from \autoref{bicatdist}.  
Let $\sA$ be an EI-category enriched in the category of abelian groups.

\begin{lem}[\autoref{specialtensor}]\label{specialtensor2}  If $\sX\colon \sA\rightarrow {\Ch_R}$
and $\sY\colon \sA^{op}\rightarrow {\Ch_R}$ are supported on isomorphisms 
\[\sX\odot \sY \cong \displaystyle{ \bigoplus_{c\in  B(\sA)}} \sX(c)\otimes_{\protect\sA(c,c)}\sY(c).\] 
\end{lem}

\begin{proof}
We will show that $\oplus \sX(c)\otimes_{\sA(c,c)}\sY(c)$ satisfies the universal
property that defines $\sX\odot \sY$.

By definition of $B(\sA)$, for any object $a$ in $\sA$ there is exactly one object 
\[c\in B(\sA)\] such that
there is an isomorphism $f\colon a\rightarrow c$ in $\sA$.   
Define a map 
\[\theta_a\colon \sX(a)\otimes_\mathbb{Z}\sY(a)\rightarrow \sX(c)\otimes_{\sA(c,c)}\sY(c)\] 
as the composite of  
\[\sX(f)\otimes \sY(f^{-1})\colon \sX(a)\otimes_\mathbb{Z}\sY(a)\rightarrow    
\sX(c)\otimes_\mathbb{Z}\sY(c)\] with the quotient map
\[\sX(c)\otimes_\mathbb{Z} \sY(c)\rightarrow \sX(c)\otimes_{\sA(c,c)}\sY(c).\]

If $g$ is another isomorphism in $\sA$ from $a$ to $c$, then 
$(\sX(f)(A),\sY(f^{-1})(B))$ is identified with $(\sX(g)(A),\sY(g^{-1})(B))$ and the 
map $\theta_a$ is well defined.
Let \[\theta\colon  \displaystyle{ \bigoplus_{a\in  \ob(\sA)}} \sX(a)\otimes_\bZ \sY(a)
\rightarrow 
\displaystyle{ \bigoplus_{c\in  B(\sA)}} \sX(c)\otimes_{\protect\sA(c,c)}\sY(c)\]
be the sum of the maps $\theta_a$.

If $(A,f,B)\in  \sX(a)\otimes_\mathbb{Z} \sA(a,b) 
\otimes_\mathbb{Z} \sY(b)$ the images of this element in  
\[{\oplus_{a\in  \ob(\sA)}} \sX(a)\otimes_\bZ \sY(a)\]
are $(A,\sY(f)(B))$ and $(\sX(f)(A),B)$.    
The images of these elements are identified under $\theta$.

Let \[\phi\colon \bigoplus_{a\in\ob\sA} 
\sX(a)\otimes_\mathbb{Z} \sY(a)\rightarrow
M\]  be a map that coequalizes the two maps from $\oplus_{a,b\in\ob\protect\sA}
 \sX(a)\otimes_\mathbb{Z}\protect\sA(a,b)
\otimes_\mathbb{Z} \sY(b)$ to $\oplus_{a\in \ob \sA} \sX(a)\otimes_\mathbb{Z} \sY(a).$  
Define a map \[\psi\colon \bigoplus_{c\in B(\sA)} \sX(c)\otimes_{\sA(c,c)}
\sY(c)\rightarrow M\] by choosing lifts of elements in  $\sX(c)\otimes_{\sA(c,c)}
\sY(c)$ to elements of \[\sX(c)\otimes_{\mathbb{Z}}\sY(c).\]  Since $\phi$ coequalizes, 
the choices do not matter and $\psi$ is unique.
\end{proof}

\begin{lem}[\autoref{explicitdual}]\label{explicitdual2} Let $\sX$ and $\sY$ satisfy 
the conditions of \autoref{specialtensor}.
If $\sX(c)$ is dualizable as a $\sA(c,c)$-module
with dual $\sY(c)$ for each $c\in B(\sA)$ then $\sX$ is dualizable with dual $\sY$.

\end{lem}

\begin{proof}
If $\sX(c)$ is dualizable as an $\sA(c,c)$-module
with dual $\sY(c)$ then there is a map of chain complexes of abelian groups
\[\eta_{c}\colon \mathbb{Z}\rightarrow \sX(c)\odot \sY(c)\] and a map of chain complexes
of $\sA(c,c)$-bimodules
\[\epsilon_{c}\colon  \sY(c)\odot \sX(c)\rightarrow \sA(c,c)\]  for each $c\in B(\sA)$. 

Let $\eta\colon \mathbb{Z}\rightarrow \sX\odot \sY$ be the composite
\[\mathbb{Z}\stackrel{\triangle}{\rightarrow} \bigoplus_{B(\sA)} 
\mathbb{Z}\stackrel{\oplus \eta_{c}}{\rightarrow}
\bigoplus_{B(\sA)} \sX(c)\otimes_{\sA(c,c)}\sY(c)\cong \sX\odot \sY\]
where $\triangle\colon \mathbb{Z}\rightarrow \oplus_{B(\sA)} \mathbb{Z}$ is the map that takes $1$ to 
$(1,1, \ldots , 1)$.

Let $a$ and $b$ be isomorphic objects of $\sA$.  Let $c$ be an object of $B(\sA)$
that is isomorphic to $a$ and let $h$ be an 
isomorphism in $\sA$ from $a$ to $c$ and $g$ be an isomorphism from $b$ to $c$. 
Then  $\epsilon_{a,b}$ is the composite
\[\xymatrix{
{\sY}(c)\otimes_\mathbb{Z}\sX(c)\ar[r]^-{\epsilon_{c}}&
\protect\sA(c,c)\ar[d]^{\protect\sA(g,h^{-1})}\\
{\sY}(b)\otimes_\mathbb{Z}\sX(a)\ar[u]^{\sY(g^{-1})\otimes \sX(h)}
&\protect\sA(b,a)}.\]   If $a$ and $b$ are not 
isomorphic in $\sA$ $\epsilon_{a,b}$ is zero.  Since $c$ is unique and 
the maps $\epsilon_c$ are maps of $\sA(c,c)$- bimodules, $\epsilon$ is a natural
transformation.  This also implies that $\epsilon$ is independent of the choice of 
$g$ and $h$.

Let $\eta_c(1)=\displaystyle\sum_{i} e_{c,i}\otimes f_{c,i}$ for each $c\in B(\sA)$.
If $x\in \sX(a)$ the value of the composite
\[\sX(a)\cong \bZ\otimes \sX(a)\stackrel{\eta\otimes 1}{\rightarrow } \sX\odot \sY \otimes \sX(a)
\stackrel{1\odot \epsilon}{\rightarrow} \sX\odot \sA(-,a)\cong \sX(a)\]  applied to $x$ is 
\[\sum_{c\in B(\sA)}\sum_{i}  \sX(\epsilon(f_{c,i},x))(e_{c,i}).\]   The only nonzero
terms in this sum are those where there is an isomorphism $h$ from $x$ to $c$.  By definition, 
$\epsilon (f_{c,i},x)=h^{-1}\epsilon_c(f_{c,i},\sX(h)(x))$ and 
\[\begin{array}{lll}\displaystyle\sum_{i}\sX(\epsilon(f_{c,i},x))(e_{c,i})&=&\sX(h^{-1})
\displaystyle\sum_{i}
\sX(\epsilon_c(f_{c,i}, \sX(h)(x)))(e_{c,i})\\
&=&\sX(h^{-1})\sX(h)(x)\\
&=&x.
\end{array}\]
The other diagram is similar.
\end{proof}